\documentclass[11pt,a4paper]{amsart}
\usepackage{times}
\usepackage{amssymb}

\usepackage{ifthen}
\usepackage{cite}
\usepackage{graphicx}

\makeatletter
\newcommand{\ga}{\alpha}
\newcommand{\gb}{\beta}
\newcommand{\gd}{\delta}
\newcommand{\gep}{\epsilon}

\newcommand{\gk}{\kappa}
\newcommand{\gl}{\lambda}
\newcommand{\go}{\omega}
\newcommand{\gs}{\sigma}

\newcommand{\vt}{\vartheta}
\newcommand{\gD}{\Delta}

\newcommand{\gL}{\Lambda}
\newcommand{\gO}{\Omega}

\newcommand{\cA}{\mathcal{A}}
\newcommand{\cB}{\mathcal{B}}
\newcommand{\cC}{\mathcal{C}}
\newcommand{\cD}{\mathcal{D}}
\newcommand{\cE}{\mathcal{E}}
\newcommand{\cF}{\mathcal{F}}
\newcommand{\cG}{\mathcal{G}}
\newcommand{\cH}{\mathcal{H}}
\newcommand{\cI}{\mathcal{I}}

\newcommand{\cL}{\mathcal{L}}
\newcommand{\cM}{\mathcal{M}}

\newcommand{\cT}{\mathcal{T}}

\newcommand{\cX}{\mathcal{X}}

\newcommand{\1}{1}

\newcommand{\B}{\mathbb{B}}
\newcommand{\C}{\mathbb{C}}

\newcommand{\N}{\mathbb{N}}
\newcommand{\Q}{\mathbb{Q}}
\newcommand{\R}{\mathbb{R}}

\newcommand{\X}{\mathbb{X}}

\newcommand{\Z}{\mathbb{Z}}
\newcommand{\p}{\partial}

\newtheorem{theorem}{Theorem}[section]
\newtheorem{proposition}[theorem]{Proposition}
\newtheorem{lemma}[theorem]{Lemma}
\newtheorem{corollary}[theorem]{Corollary}

\newtheorem{definition}[theorem]{Definition}

\theoremstyle{remark}

\newtheorem{remark}[theorem]{Remark}
\newtheorem{remarks}[theorem]{Remarks}

\newcommand{\diag}[1]{\text{diag}(#1)}
\newcommand{\erf}{\mathop{\operator@font erf}\nolimits}
\newcommand{\erfc}{\mathop{\operator@font erfc}\nolimits}
\newcommand{\sign}{\mathop{\operator@font sign}\nolimits}

\newenvironment{enum_a}
    {\begin{enumerate}}
    {\end{enumerate}}
\newenvironment{enum_i}
    {\begin{enumerate}}
    {\end{enumerate}}

\newif\if@golden  \@goldentrue
\newcommand{\f@ctor}{1}
\newlength{\aiv@width}  
\newlength{\aiv@height} 
\newlength{\tmp@width}  
\newlength{\tmp@height} 
\if@twocolumn\if@golden
  \else\fi
\else\if@golden\ifcase\@ptsize\relax\or
\or\fi
  \else\ifcase\@ptsize\relax\or
\or\fi\fi
\if@twocolumn\if@golden\ifcase\@ptsize\relax\renewcommand{\f@ctor}{53}
  \or\renewcommand{\f@ctor}{46}\or\renewcommand{\f@ctor}{43}\fi
  \else\ifcase\@ptsize\relax\renewcommand{\f@ctor}{51}\or
  \renewcommand{\f@ctor}{45}\or\renewcommand{\f@ctor}{42}\fi\fi
\else\if@golden\ifcase\@ptsize\relax \renewcommand{\f@ctor}{46}
  \or\renewcommand{\f@ctor}{43}\or\renewcommand{\f@ctor}{43}\fi
  \else\ifcase\@ptsize\relax\renewcommand{\f@ctor}{43}
  \or\renewcommand{\f@ctor}{40}\or\renewcommand{\f@ctor}{40}\fi\fi\fi
\multiply\textheight by \f@ctor
\let\comp\circ

\newcommand{\op}{\ensuremath^\circ}

\newcommand{\cGo}{\ensuremath\cG\op}

\newcommand{\sgl}{\ensuremath\sqrt{2\gl}}

\newcommand{\da}{\ensuremath\downarrow}

\newcommand{\cond}{\ensuremath\,\big|\,}
\newcommand{\Co}{\ensuremath C_0(\cG)}
\newcommand{\Coo}{\ensuremath C_0^{0,2}(\cG)}
\newcommand{\Cii}{\ensuremath C_0^2(\cG)}
\newcommand{\Coe}{\ensuremath C_0(E)}
\newcommand{\Ra}{$\Rightarrow$\space}
\newcommand{\limep}{\ensuremath\lim_{\gep\da 0}}

\newcommand{\eval}{\mathop{\big|}\nolimits}
\newcommand{\Eval}{\mathop{\Big|}\nolimits}

\newcommand{\CD}{\ensuremath C_{\gD}}

\newcommand{\Ieqref}[1]{\textup{\tagform@{I.\ref{I_#1}}}}

\newcommand{\IIeqref}[1]{\textup{\tagform@{II.\ref{II_#1}}}}

\newcommand{\Rbp}{\ensuremath\overline{\R}_+}

\newcommand{\Xil}[1][n-1]{\ensuremath\Xi^{\le #1}}
\newcommand{\cCl}[1][n-1]{\ensuremath\cC^{\le #1}}
\newcommand{\Ql}[1][n-1]{\ensuremath Q^{\le #1}}

\newcommand{\Xiu}[1][n]{\ensuremath\Xi^{\ge #1}}
\newcommand{\cCu}[1][n]{\ensuremath\cC^{\ge #1}}
\newcommand{\Qu}[1][n]{\ensuremath Q^{\ge #1}}
\newcommand{\bcB}{\ensuremath B}
\newcommand{\bRV}{\ensuremath B(\R_+\times V_c)}
\newcommand{\bcBG}[1][r]{\ensuremath B(\cG^#1)}
\newcommand{\Rrp}[1][r]{\ensuremath \hat \R^{#1}_+}
\newcommand{\sh}{\ensuremath\text{shad}}
\newlength{\BCs@ze}
\newlength{\BCsh@ft}
\ifcase\@ptsize\relax
\or
\or
\fi
\DeclareFixedFont\MT{OMS}{cmsy}{m}{n}{\BCs@ze}    
\newcommand{\BigCart}{\ensuremath\mathop{\raisebox{\BCsh@ft}{{\MT\char"02}}}}

\def\pdftitle{\@gobble}

\makeatother

\numberwithin{equation}{section}
\allowdisplaybreaks[2]

\date{April 20, 2012}
\dedicatory{Dedicated to Elliott Lieb on the occasion of his 80th birthday}

\title[Brownian Motions on Metric Graphs]{%
Brownian Motions on Metric Graphs}

\author[V.~Kostrykin]{Vadim Kostrykin}
\address{Vadim Kostrykin\newline
Institut f\"ur Mathematik\newline
Johannes Gutenberg--Universit\"at\newline
D--55099 Mainz, Germany}
\email{kostrykin@mathematik.uni-mainz.de}

\author[J.~Potthoff]{J\"urgen Potthoff}
\address{J\"urgen Potthoff\newline
Institut f\"ur Mathematik, Universit\"at Mannheim\newline
D--68131 Mann\-heim, Germany}
\email{potthoff@math.uni-mannheim.de}

\author[R.~Schrader]{Robert Schrader}
\address{Robert Schrader\newline
Institut f\"{u}r Theoretische Physik\newline
Freie Universit\"{a}t Berlin, Arnimallee~14\newline
D--14195 Berlin, Germany}
\email{schrader@physik.fu-berlin.de}

\copyrightinfo{2011}{V.~Kostrykin, J.~Potthoff, R.~Schrader}

\begin{document}
\begin{abstract}
Brownian motions on a metric graph are defined. Their generators are	
characterized as Laplace operators subject to Wentzell boundary at every
vertex. Conversely, given a set of Wentzell boundary conditions at the
vertices of a metric graph, a Brownian motion is constructed pathwise on
this graph so that its generator satisfies the given boundary conditions.
\end{abstract}

\maketitle
\thispagestyle{empty}

\section{Introduction and Main Results} \label{sect_1}

Since the groundbreaking works of Bachelier~\cite{Ba00}, Einstein~\cite{Ei05, Ei06},
and Smo\-lu\-chow\-ski~\cite{Sm06},%
\footnote{\footnotesize It seems that Schr\"odinger \cite{Sc15} was the first
to introduce the notion of a \emph{first passage time}, (in German \emph{Erstpassagezeit}),
i.e., a special type of \emph{stopping time}, in the continuous time context
of the Brownian motion process. It is striking that this article and the parallel work of
Smo\-lu\-chow\-ski~\cite{Sm15} has practially gone unnoticed in the physics literature,
while  being cited by statisticians, e.g., \cite{Tw45, FoCh78}.}\space
the theory of the Brownian movement had been established as a central,
recurrent theme in mathematics and physics. In the sequel the Brownian phenomenon stimulated the
development of many important ideas and theories. A complete description of the history is
beyond the scope of this introduction, but in keywords  we want to mention the following:	
The construction of Wiener space~\cite{Wi23} and Wiener's approach of statistical mechanics
and chaos~\cite{Wi38}, It\^o's theory of stochastic integration~\cite{It44} and stochastic
differential equations~\cite{It46}, L\'evy's analysis of the fine structure of
Brownian motion and his theory of the Brownian local time~\cite{Le37, Le48},
Feynman's path integral~\cite{Fe48} with its new view towards quantum mechanics,
Kac' work on path integrals~\cite{Ka49, Ka50}. Towards the middle of the last century
there were the works by Feller~\cite{Fe52, Fe54, Fe54a} and It\^o--McKean~\cite{ItMc63, ItMc74} on
Brownian motions on intervals (see also below), Gross' abstract Wiener spaces~\cite{Gr67},
Nelson's work on functional integration and on the relation between quantum and
stochastic dynamics~\cite{Ne54, Ne64, Ne67, Ne73}, giving new momentum to Euclidean and
to constructive quantum field theory, e.g.~\cite{Sch58, Sch59, Sy69, Si74, GlJa81} and
nonrelativistic quantum physics, e.g.~\cite{Si79}. Further we want to mention  the asymptotics
of Wiener integrals and large deviation theory~\cite{Sc65, DoVa75}, the theory of Dirichlet forms~\cite{Fu80, Si75, AlR89}, the development of the
Malliavin~\cite{Ma78, Ma97} and Hida calculi~\cite{Hi75, HiKu93}, and Bismuth's approach
to the Atiyah--Singer index theorem~\cite{Bi84, Bi86}. In addition, there were important
developments in other fields, such as engineering, biology or mathematical finance, which
were triggered by the theory of Brownian motion.

The present article is directly linked to the above quoted works by Feller and It\^o--McKean.
So we want to sketch these in little more detail.
In his pioneering articles~\cite{Fe52, Fe54, Fe54a}, Feller
raised the problem of characterizing and constructing all Brownian motions on a finite or on a
semi-infinite interval. In the sequel this problem stimulated very important
research in the field of stochastic processes, and the problem of constructing all
such Brownian motions found a complete solution in the work of It\^o and
McKean~\cite{ItMc63, ItMc74} via the combination of the theory of the local time of
Brownian motion~\cite{Le48}, and the theory of (strong) Markov
processes~\cite{Bl57, Dy61, Dy65a, Dy65b, Hu56}.
The central result of these investigations is that the most general Brownian motion on
the half line $\R_+$ is determined by a generator which is (one half times) the Laplace
operator on $\R_+$ with \emph{Wentzell boundary conditions} at the origin, i.e.,
linear combinations of the function value with the values of the first and second
derivative at the origin (with coefficients satisfying certain restrictions, see below).
It\^o and McKean showed in~\cite{ItMc63} --- partly based on the ideas of
Feller~\cite{Fe52, Fe54, Fe54a} --- how to construct the paths of such
motions: The boundary conditions are implemented by a combination of reflection
at the origin with a slow down and killing, both on the scale of the local time
at zero. The ideas contained this article became one of the roots of their
highly influential book~\cite{ItMc74}.

In recent years, there has been a growing interest in \emph{metric graphs}, that is,
piecewise linear spaces with singularities formed by the vertices of the
graph. Metric graphs arise naturally as models in many domains, such as
physics, chemistry, computer science and engineering to mention just a few --- we refer
the interested reader to~\cite{Ku04} for a review of such models and for further
references. Therefore it is natural to extend Feller's problem to metric graphs.
Stochastic processes, in particular Brownian motions and diffusions, on locally
one-dimensional structures, notably on graphs and networks, have already been studied
in a number of articles of which we want to mention~\cite{BaCh84, DeJa93, EiKa96, FrSh00,
FrWe93, Fr94, Gr99, Kr95} in this context.

In previous articles~\cite{KoSc99, KoSc00, KoSc06, KoSc06a}, two of the current authors
studied the self-ad\-joint\-ness of Laplace operators on metric graphs and discussed their spectra.
This allowed a discussion of the associated quantum scattering matrices.
Further properties of the semigroups generated by Laplace operators on metric graphs, including
a Selberg--type trace formula and the problem whether these semigroups are
positivity preserving or contractive, have been studied in \cite{KoPo07d, KoPo09c}.
This is one of the motivations of our present study since semigroups with these properties
typically show up in Markov processes.  Below we will return to this point,
see remark~\ref{rem_sa_bc}. The wave equation on metric graphs and its finite propagation
speed has been discussed in~\cite{KoPo11}. For suitable Laplacians free quantum fields on
metric graphs satisfying the Klein-Gordon equation and Einstein causality were
constructed in ~\cite{Sch09}.

In~\cite{CPBMSG} the authors have constructed the paths of all possible Brownian
motions (in the sense defined below) on single vertex graphs using the well-known Walsh
process \cite{Wa78}, \cite{BaPi89} (see also \cite{ BaCh84, Ro83, Sa86a, Sa86b, Va85})
as the starting point. Furthermore, the relation to the quantum mechanical scattering is
discussed in detail there. The latter article provides an essential input for the
construction of all possible Brownian motions on a general metric graph in the sense
of definition~\ref{def1i} (see below) which we carry out here.

The article is organized in the following way. In section~\ref{sect_mr} we set	
up our framework and prove our main results: Theorem~\ref{thm1i} characterizes
all possible Brownian motions (in the sense of definition~\ref{def1i}) on a metric
graph~$\cG$ in terms of Wentzell boundary conditions at the vertices. Conversely,
theorem~\ref{thm1ii} states that for every choice of a set of Wentzell boundary
conditions at the vertices as described in theorem~\ref{thm1i}, one can construct
a Brownian motion on~$\cG$ implementing these conditions. Theorem~\ref{thm1i} is
proved in section~\ref{sect2}. As a preparation of the proof of theorem~\ref{thm1ii}
we consider in section~\ref{sect3} the situation where one is given two
metric graphs $\cG_1$, $\cG_2$ with Brownian motions $X_2$, $X_2$ in the sense of
definition~\ref{def1i} thereon. If one joins some of the external edges of $\cG_1$ and
$\cG_2$ to form a new metric graph $\cG$, it is shown how to construct the
paths of a Brownian motion $X$ on $\cG$ by appropriately gluing the paths of
$X_1$ and $X_2$ together. Theorem~\ref{thm1ii} is proved in section~\ref{sect4}
via the procedure of section~\ref{sect3} and the results in~\cite{CPBMSG},
where the paths of Brownian motions on star graphs are constructed with
methods similar to those of Feller~\cite{Fe52, Fe54, Fe54a} and
It\^o--McKean~\cite{ItMc63, ItMc74}. The article is concluded in section~\ref{sect5}
by a discussion of the inclusion of tadpoles. Furthermore, there two appendices: one with
a technical result on the crossover times which is used in section~\ref{sect3}, the other
about Feller semigroups and resolvents.

Given these results, it would be interesting to see whether known results
for special cases of Brownian motion or diffusions on metric graphs
can be extended to all Feller processes. For example, an arcsine law has been proved
in~\cite{BaPi89a} for the case of a Walsh process on a single vertex graph, for the case
of a general metric graph we refer to~\cite{De02, BDe09} (for a discussion of local time
distributions see also~\cite{CoDe02}). In a similar vein: What about occupation times on
edges for the case of general (local) boundary conditions of the
type~\eqref{eq1iii} at the vertices? Can one say something about large deviations
as done for example for Brownian motions without killing and more generally for
conservative diffusion processes in~\cite{FrSh00}? What form does the It\^o formula
take in the case of a diffusion process on a metric graph with a generator subject to
the boundary conditions~\eqref{eq1iii}?

\vspace{1.5\baselineskip}
\noindent
\textbf{Acknowledgement.}
The authors thank Mrs.~and Mr.~Hulbert for their warm hospitality at the
\textsc{Egertsm\"uhle}, Kiedrich, where part of this work was done.
J.P.\ gratefully acknowledges fruitful discussions with O.~Falkenburg,
A.~Lang and F.~Werner. We owe special thanks to O.~Falkenburg	
for pointing out reference~\cite{Sc15} to us. The authors also thank the anonymous
referee for pointing out further references. R.S.~thanks the organizers of  the
\emph{Chinese--German Meeting on Stochastic Analysis and Related Fields},  Beijing,
May 2010, where some of the material of this article was presented.

\section{Main Results}	\label{sect_mr}

In the present article we shall only treat \emph{finite} metric graphs, and consider a
metric graph $(\cG,d)$ as being defined by a finite collection of finite or semi-infinite
closed intervals, some of their endpoints --- the \emph{vertices} of the graph --- being
identified. See figure~\ref{fig1} for an example of a simple, typical metric graph.
The metric $d$ is then defined in the canonical way as the length of a shortest path
between two points along the \emph{edges} (formed by the intervals),
and the length along each edge is measured with the usual metric on the real line.
\begin{figure}[ht]
\begin{center}
    \includegraphics[scale=1]{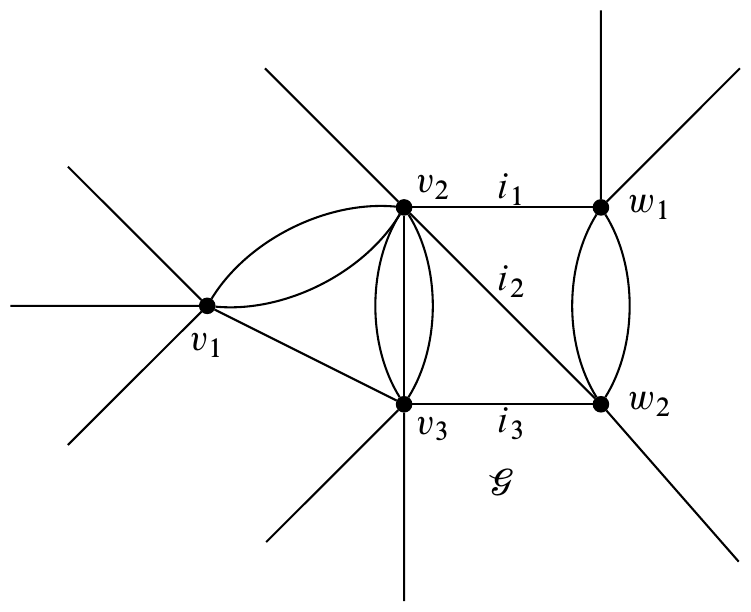}
    \caption{A metric graph $\cG$ with $5$ vertices, $9$ external and
                $11$ internal edges.}  \label{fig1}
\end{center}
\end{figure}
For a formal definition of metric graphs within the context of graph theory we
refer the interested reader, e.g., to~\cite{KoSc99, KoSc00}. Within that context our
definition above means that we identify --- as we may without any loss of generality
--- an abstract metric graph with its \emph{geometric graph} (see, e.g., \cite{Ju05}).
Moreover, in the sequel it will often be convenient and without any danger of confusion
to identify an edge of a metric graph with the corresponding interval of the real line.
Edges isomorphic to $\R_+$ are called \emph{external}, while those isomorphic to a
finite interval --- that is, those edges connecting two vertices --- are called
\emph{internal}. The set of vertices of $\cG$ is denoted by $V$, the set of
internal edges by $\cI$ and the set of external edges by $\cE$. Moreover we
set $\cL=\cI\cup\cE$. The combinatorial structure of the graph $\cG$ is described
by a map $\gd$ from  $\cL$ into $V\cup (V\times V)$ which associates with every
internal edge $i$ an ordered pair $\bigl(\p^-(i), \p^+(i)\bigr)\in V\times V$,
$\p^-(i)$ is called the \emph{initial vertex} of $i$ while $\p^+(i)$ is its
\emph{terminal vertex}. If $i\in\cI$ is isomorphic to the interval $[a,b]$ then $\p^-(i)$
corresponds to $a$, while $\p^+(i)$ corresponds to~$b$. An external edge $e$ is mapped
by $\p$ to $\p(e)\in V$ which is the vertex to which $e$ is incident, and also in this case
we call the vertex the \emph{initial vertex} of~$e$.

For the definition of a Brownian motion on the metric graph $(\cG,d)$ we take a
standpoint similar to the one of Knight~\cite{Kn81} for the semi-line or a finite
interval:

\begin{definition}  \label{def1i}
A \emph{Brownian motion} on a metric graph $(\cG,d)$ is a diffusion process
$(X_t,\,t\in\R_+)$ such that when $X$ starts on an edge $e$ of $\cG$ then the
process $X$ with absorption in the vertex, vertices respectively, to which $e$ is
incident is equivalent to a standard one dimensional Brownian motion on the interval
$e$ with absorption in the endpoint(s) of $e$.
\end{definition}

\begin{remarks}  \label{def1ii}
By saying that $X$ is a \emph{diffusion process} we mean that $X$ is a normal,
strong Markov process (in the sense of~\cite{BlGe68}), a.s.\ with paths which are
c\`adl\`ag and continuous on $[0,\zeta)$, where $\zeta$ is the life time of $X$.
We shall always assume that the filtration for $X$ satisfies the ``usual conditions''.
With the help of the well-known first passage time formula (e.g., \cite{Ra56} or
\cite{ItMc74}) for the resolvent of $X$ it is not hard to show as in~\cite{Kn81}
that every Brownian motion on a metric graph $\cG$ is a Feller process.
\end{remarks}

The first crucial problem is then to characterize the behavior of the stochastic process
when it reaches one of the vertices of the graph $\cG$, or in other words, the
characterization of the boundary conditions at the vertices of the Laplace operator
which generates the stochastic process. We want to mention in passing that in an
$L^2$-setting all boundary conditions for Laplace operators on $\cG$ which make them
self-adjoint have been characterized in~\cite{KoSc99, KoSc00}. The first main
result of the present paper is Feller's theorem for metric graphs. In order to state
this theorem we have to introduce some notation.

The Banach space of real valued, continuous functions on $\cG$ vanishing at infinity,
equipped with the sup-norm, is denoted by $C_0(\cG)$. We let $\gD$ denote a universal
cemetery point for all stochastic processes considered, and make the usual convention
that every $f\in C_0(\cG)$ is extended to $\cG\cup\{\gD\}$ by setting $f(\gD)=0$.

Consider the generator $A$ of $X$ on $C_0(\cG)$ with domain $\cD(A)$. Define the space
$\Cii$ to consist of those functions $f$ in $C_0(\cG)$ which are twice continuously
differentiable in the \emph{open interior} $\cG^\circ = \cG\setminus V$ of $\cG$, and
which are such that their second derivative $f''$ extends from $\cG^\circ$ to a
function in $C_0(\cG)$.

The next lemma, which can be proved with the fundamental theorem of calculus and the
mean value theorem, states some of the properties of functions in $\Cii$.
$\cL(v)$ denotes the set of edges incident with $v\in V$.

\begin{lemma}		\label{lemC02}
Assume that $f$ belongs to $\Cii$, and consider $v\in V$, $l\in\cL(v)$. Then the inward directional
derivatives $f^{(i)}(v_l)$, $i=1$, $2$, of $f$ of first and second order at $v$ in
direction of the edge $l$ exist, and
\begin{align}	
	f'(v_l) &= \begin{cases}\displaystyle
				\phantom{-}\lim_{\xi\to v,\,\xi\in l^\circ} f'(\xi),	
					& \text{if $v$ is an initial vertex of $l$,}\label{inw-deri}\\[2ex]
				\displaystyle
				-\lim_{\xi\to v,\,\xi\in l^\circ} f'(\xi),	
					& \text{if $v$ is a terminal vertex of $l$,}
		      \end{cases}\\[2ex]
	f''(v_l) &= \lim_{\xi\to v,\,\xi\in l^\circ} f''(\xi)\label{inw-derii} 		
\end{align}
hold true. Moreover, $f'$ (defined on $\cG^\circ$) vanishes at infinity.
\end{lemma}

\begin{remark}	\label{remC02}
If $f\in\Cii$ then by definition of $\Cii$, $f''(v_k) = f''(v_l)$ for every
$v\in V$, and all $k$, $l\in\cL(v)$, and we shall simply write $f''(v)$.
On the other hand, in general $f'(v_k)\ne f'(v_l)$ for $k\ne l$.
\end{remark}

Let $V_\cL$ denote the subset of $V\times\cL$ given by
\begin{equation*}
    V_\cL = \bigl\{(v,l),\,v\in V \text{ and }l\in\cL(v)\bigr\}.
\end{equation*}
We shall also write $v_l$ for $(v,l)\in V_\cL$. Consider data
of the following form
\begin{equation}    \label{eq1i}
\begin{split}
    a &= (a_v,\,v\in V)\in [0,1)^V\\
    b &= (b_{v_l},\,v_l\in V_\cL) \in [0,1]^{V_\cL}\\
    c &= (c_v,\,v\in V)\in [0,1]^V
\end{split}
\end{equation}
subject to the condition
\begin{equation}    \label{eq1ii}
    a_v + \sum_{l\in \cL(v)} b_{v_l} + c_v =1,\qquad \text{for every $v\in V$}.
\end{equation}
Define a subspace $\cH_{a,b,c}$ of $\Cii$ as the space of those functions
$f$ in $\Cii$ which at every vertex $v\in V$ satisfy the \emph{Wentzell boundary
condition}
\begin{equation}\label{eq1iii}
    a_v f(v) - \sum_{l\in\cL(v)} b_{v_l}f'(v_l)+ \frac{1}{2}\,c_v f''(v)=0.
\end{equation}

Now we can state our first main result:

\begin{theorem}[Feller's theorem for metric graphs] \label{thm1i}
Let $X$ be a Brownian motion on $\cG$, and let $A$ be its generator on $C_0(\cG)$
with domain $\cD(A)$. Then there are $a$, $b$, $c$ as in~\eqref{eq1i}, \eqref{eq1ii},
so that $\cD(A)=\cH_{a,b,c}$. For $f\in\cD(A)$, $A f = 1/2 f''$.
\end{theorem}

\begin{remark}
The boundary conditions in $\cH_{a,b,c}$ are \emph{local} in the
sense that only \emph{one vertex} enters each of the conditions~\eqref{eq1iii}. This
is a direct consequence of the path properties of $X$, namely of the condition
that the only jumps $X$ may have are those from $\cG$ (actually from a vertex)
to the cemetery point.
\end{remark}

\begin{remark}\label{standard}
Boundary conditions with $a_v=0=c_v$ are often called \emph{standard boundary conditions}
(see e.g.~\cite{Ku04, KoSc06})
giving rise to what is called a \emph{skew Brownian motion} \cite{ItMc74}, for a recent survey
see e.g.~\cite{Le06}. Killing occurs when
$a_v\neq 0$, and~\cite{CPBMSG} provides a detailed discussion of the process for
single vertex graphs. When $a_v= 0$ the process is conservative and has been studied
extensively in~\cite{FrWe93,FrSh00}.
\end{remark}

Our second main result is converse of theorem~\ref{thm1i}, namely

\begin{theorem} \label{thm1ii}
For any choice of the data as in~\eqref{eq1i}, \eqref{eq1ii}, there is a Brownian
motion $X$ on the metric graph $\cG$ so that its generator $A$ has $\cH_{a,b,c}$ as
its domain.
\end{theorem}

In order to rephrase the statements of theorems~\ref{thm1i} and~\ref{thm1ii} in
a concise way, we bring in some additional notation. With a slight abuse of language
we shall also call any quadruple $\X=(\gO,\cA,P,X)$ a \emph{Brownian motion on $\cG$}
whenever $(\gO,\cA,P)$ is a complete probability space, and $X=(X_t,\,t\in\R)$ defined
thereon is a Brownian motion on $\cG$ as in definition~\ref{def1i}.
$\cX(\cG)$ denotes the set of all Brownian motions in this sense, subject to the
equivalence relation which is defined by equality of all finite dimensional
distributions. In $\R^{n+1}$ consider the (compact, convex) $n$-simplex	
\begin{equation*}
	\sigma^n= \Bigl\{x\in\R^{n+1},\, x_i\ge 0, \sum_{i=1}^{n+1}x_i=1\Bigr\}.
\end{equation*}
Let $\sigma^n_0$ be the simplex $\sigma^n$, with the point $(1,0,\cdots,0)$ removed
\begin{equation*}
	\sigma^n_0=\sigma^n\setminus\{(1,0,\cdots,0)\}.
\end{equation*}
It is still convex but not closed. Given a fixed but arbitrary ordering of
$\mathcal{L}(v)$ any triple $\bigl(a_v,(b_{v_l}, \,l\in\mathcal{L}(v)),c_v\bigr)$
satisfying~\eqref{eq1ii} can be viewed as an element in $\sigma^{n(v)+1}_0$ with
$n(v)=|\mathcal{L}(v)|$. With $N(\mathcal{G})=\sum_v( n(v) +2)$ set
\begin{equation*}
	\Sigma(\mathcal{G})= \BigCart_{v\in V}
						\sigma^{n(v)+1}_0\subset \R^{N(\mathcal{G})}.
\end{equation*}
Let $\iota$ be the mapping defined via theorem~\ref{thm1i} by
associating to every Brownian motion $\X$ the data $(a,b,c)\in\Sigma(\cG)$.
Since any two Brownian motions on $\cG$, which have the same finite dimensional
distributions, define the same semigroup, and therefore have the same generator,
it follows that $\iota$ maps these to the same data, that is, $\iota$
can be viewed as mapping from $\cX(\cG)$ to $\Sigma(\cG)$. Theorem~\ref{thm1ii}
states that $\iota$ is surjective. To see its injectivity, suppose that
$[\X_1]$, $[\X_2]$ are different elements in $\cX(\cG)$, where $[\X_i]$,
$i=1$, $2$, denotes the equivalence class of a representative $\X_i$. Assume
that $\iota([\X_1]) = \iota([\X_2])$. By theorem~\ref{thm1i} the generator
$A_i$, $i=1$, $2$, of $\X_i$ is uniquely determined by the data
$\iota([\X_i])$, and therefore we get $A_1=A_2$. It follows, that
$\X_1$ and $\X_2$ define the same semigroup, and therefore all their
finite dimensional distributions coincide, which is a contradiction. Thus we have
proved that $\iota$ is a bijection from $\cX(\cG)$ onto $\Sigma(\cG)$:

\begin{corollary}
The set $\cX(\cG)$ of all Brownian motions on $\mathcal{G}$ is in one-to-one
correspondence with the set $\Sigma(\cG)$.
\end{corollary}

\section{Proof of Theorem~\ref{thm1i}}	\label{sect2}
The following notation will be useful throughout this article: If $\xi$ is
a point in $\cG^\circ = \cG\setminus V$, then it is in one-to-one
correspondence with its \emph{local coordinates} $(l,x)$, where $l\in\cL$
is the edge to which $\xi$ belongs, while $x$ is the point corresponding
to $\xi$ in the interval to which $l$ is isomorphic. Then we simply write
$\xi=(l,x)$. If $f$ is a function on the graph $\cG$ we shall also denote
$f(\xi)$ by $f(l,x)$ or $f_l(x)$.

We denote by $U=(U_t,\,t\in\R_+)$ the semigroup generated by a Brownian motion $X$ on
$\cG$ acting on the Banach space $B(\cG)$ of bounded measurable functions on $\cG$,
equipped with the sup-norm, that is, for $f\in\B(\cG)$,
\begin{equation*}
	U_t f(\xi) = E_\xi\bigl(f(X_t)\bigr),\qquad t\in\R_+,\,\xi\in\cG.
\end{equation*}
Clearly, $U$ is a positivity preserving contraction semigroup.
In the sequel we shall notationally not distinguish between the semigroup $U$
acting on $B(\cG)$ and its restriction to the subspace $\Co$ of $B(\cG)$.

The proof of the following lemma can be taken over with minor modifications from
the standard literature, e.g., from~\cite[Chapter~6.1]{Kn81}. Therefore it is
omitted here.

\begin{lemma}   \label{lem2i}
For every Brownian motion $X$ on the metric graph $\cG$, the generator $A$ of its
semigroup $U$ acting on $\Co$ has a domain $\cD(A)$ contained in $\Cii$. Moreover,
for every $f\in\cD(A)$, $A f=1/2\,f''$.
\end{lemma}

The preceding lemma implies the second statement of theorem~\ref{thm1i}.
The proof of the first statement of theorem~\ref{thm1i} has two rather
distinct parts, and therefore we split it by proving the following two lemmas:

\begin{lemma}   \label{lem2ii}
Suppose that $X$ is a Brownian motion on a metric graph $\cG$, and that $\cD(A)$
is the domain of the generator $A$ of its semigroup. Then there are $a$, $b$, $c$
as in~\eqref{eq1i}, \eqref{eq1ii}, so that $\cD(A)\subset\cH_{a,b,c}$.
\end{lemma}

\begin{lemma}   \label{lem2iii}
Suppose that $A$ is the generator of a Brownian motion $X$ on $\cG$ with domain
$\cD(A)\subset \cH_{a,b,c}$ for some $a$, $b$, $c$ as in~\eqref{eq1i},
\eqref{eq1ii}. Then $\cD(A)=\cH_{a,b,c}$
\end{lemma}

\begin{proof}[Proof of lemma~\ref{lem2ii}]
Our proof follows the one in~\cite[Chapter~6.1]{Kn81} quite closely --- actually, it
is sufficient to consider a special case of the proof given there.

We show that for every vertex $v\in V$ there are constants $a_v\in [0,1)$,
$b_{v_l}\in [0,1]$, $l\in\cL(v)$, $c_v\in [0,1]$ satisfying~\eqref{eq1ii},
and such that all $f$ in the domain $\cD(A)$ of the generator satisfy the boundary
condition~\eqref{eq1iii}. To this end, we let $f\in\cD(A)$, fix a vertex $v\in V$,
and compute $A f(v)$.

Consider the exit time from $v$, i.e., the stopping time $S_v =
H(\cGo)$, where for any subset $M\subset \cG$, $H(M)\equiv H_M$ denotes the hitting time
of $M$. It is well known (e.g., \cite{Kn81, ReYo91, DyJu69}) that because of the
strong Markov property of $X$, $S_v$ is under $P_v$ exponentially distributed with a
rate $\gb_v\in[0,+\infty]$. Consequently we discuss three cases:

\vspace{.5\baselineskip}\noindent
\emph{Case $\gb_v=0$}: $X$ is absorbed at $v$, i.e., $v$ is a \emph{trap}. Thus
$U_t f(v) = f(v)$ for all $t\ge 0$. Consequently, $A f(v)=0$, and therefore
$1/2\, f''(v)=0$. Thus $f$ satisfies the boundary condition~\eqref{eq1iii} at $v$
with $a_v=0$, $c_v=1$, and $b_{v_l}=0$ for all $l\in\cL(v)$.

\vspace{.5\baselineskip}\noindent
\emph{Case $0<\gb_v<+\infty$}: In this case the process stays at $v$ $P_v$--a.s.\ for
a strictly positive, finite moment of time, i.e., $v$ is \emph{exponentially
holding}. It is well known (cf., e.g., \cite[p.~154]{Kn81}, \cite[p.~104, Prop.~3.13]{ReYo91})
that then the process has to leave $v$ by a jump, and by our assumption of path continuity on
$[0,\eta)$, the process has to jump to the cemetery $\gD$.
Therefore we get for $t>0$, $U_t f(v) = \exp(-\gb t) f(v)$, and thus $A f(v) + \gb
f(v) = 0$, and the boundary condition~\eqref{eq1iii} holds for the choice
\begin{equation}    \label{eq2i}
   a_v = \frac{\gb}{1+\gb},\quad c_v=\frac{1}{1+\gb},\quad
            b_{v_l} = 0,\,l\in\cL(v).
\end{equation}

\vspace{.5\baselineskip}\noindent
\emph{Case $\gb_v=+\infty$}: In this case the $X$ leaves the vertex $v$ immediately,
and it begins a Brownian excursion into one of the edges incident with the vertex
$v$. In particular, $v$ is not a trap. Therefore we may compute $A f(v)$
in Dynkin's form, e.g., \cite[p.~140, ff.]{Dy65a}, \cite[p.~99]{ItMc74}. For $\gep>0$
let $H_{v,\gep}$ denote the hitting time of the complement of the ball
$B_\gep(v)$ of radius $\gep$ around $v$. Then
\begin{equation}    \label{eq2ii}
   A f(v) = \limep \frac{E_v\Bigl(f\bigl(X(H_{v,\gep})\bigr)\Bigr)
                                                -f(v)}{E_v(H_{v,\gep})}.
\end{equation}
Now
\begin{align*}
    E_v\Bigl(f\bigl(X(H_{v,\gep})\bigr)\Bigr)
        &= \sum_{l\in\cL(v)} f_l(\gep)\,P_v\bigl(X(H_{v,\gep})\in l\bigr)
            + f(\gD)\,P_v\bigl(X(H_{v,\gep})=\gD\bigr)\nonumber\\
        &= \sum_{l\in\cL(v)} f_l(\gep)\,P_v\bigl(X(H_{v,\gep})\in l\bigr),
\end{align*}
where the last equality follows from $f(\gD)=0$. Let us denote
\begin{align*}
    r_l(\gep)   &= \frac{P_v\bigl(X(H_{v,\gep})\in l\bigr)}{E_v(H_{v,\gep})},
                        \ l\in\cL(v),\quad
    r_\gD(\gep) = \frac{P_v\bigl(X(H_{v,\gep})=\gD\bigr)}{E_v(H_{v,\gep})},\\[1ex]
    K(\gep)     &= 1 + r_\gD(\gep) + \gep \sum_{l\in\cL(v)} r_l(\gep).
\end{align*}
The continuity of the paths of $X$ up to the lifetime $\zeta$ yields
\begin{equation*}
    \sum_{l\in\cL(v)} P_v\bigl(X(H_{v,\gep})\in l\bigr)
                            + P_v\bigl(X(H_{v,\gep})=\gD\bigr)=1,
\end{equation*}
and therefore equation~\eqref{eq2ii} can be rewritten as
\begin{equation*}
    \limep\Bigl(A f(v) + r_\gD(\gep) f(v)
        - \sum_{l\in\cL(v)} r_l(\gep) \bigl(f_l(\gep)-f(v)\bigr)\Bigr)=0.
\end{equation*}
Since for all $\gep>0$, $K(\gep)^{-1}\le 1$, it follows that
\begin{equation*}
    \limep\Bigl(\frac{1}{K(\gep)}\,A f(v) + \frac{r_\gD(\gep)}{K(\gep)}\,f(v)
        - \sum_{l\in\cL(v)} \frac{\gep\, r_l(\gep)}{K(\gep)}\,
                \frac{f_l(\gep)-f(v)}{\gep}\Bigr)=0,
\end{equation*}
which by lemma~\ref{lem2i} we may rewrite as
\begin{equation*}
    \limep\Bigl(a_v(\gep) f(v) + \frac{1}{2}\,c_v(\gep) f''(v)
        - \sum_{l\in\cL(v)} b_{v_l}(\gep)\,\frac{f_l(\gep)-f(v)}{\gep}\Bigr)=0,
\end{equation*}
where we have introduced the non-negative quantities
\begin{equation*}
    a_v(\gep)       = \frac{r_\gD(\gep)}{K(\gep)},\quad
    c_v(\gep)       = \frac{1}{K(\gep)},\quad
    b_{v_l}(\gep)   = \frac{\gep\, r_l(\gep)}{K(\gep)},\ l\in\cL(v).
\end{equation*}
Observe that for every $\gep>0$,
\begin{equation*}
    a_v(\gep) + c_v(\gep) + \sum_{l\in\cL(v)} b_{v_l}(\gep) =1.
\end{equation*}
Therefore every sequence $(\gep_n,\,n\in\N)$ with $\gep_n>0$ and $\gep_n\da 0$ has a
subsequence so that $a_v(\gep)$, $c_v(\gep)$ and $b_{v_l}(\gep)$, $l\in\cL(v)$,
converge along this subsequence to numbers $a_v$, $c_v$, and $b_{v_l}$ respectively in
$[0,1]$, and the relation~\eqref{eq1ii} holds true. It is not hard to check that
for every $f\in\Cii$
\begin{equation*}
    \frac{f_l(\gep)-f(v)}{\gep}
\end{equation*}
converges with $\gep\da 0$ to $f'(v_l)$, and therefore we obtain that for every
vertex $v\in V$, $f\in\cD(A)$ satisfies the boundary condition~\eqref{eq1iii} with data
$a$, $b$, $c$ as in~\eqref{eq1i}, \eqref{eq1ii}.
\end{proof}

Before we can prove lemma~\ref{lem2iii} we have to introduce some additional
formalism.

We define the subspace $\Coo$ of functions $f$ in $\Co$ which are twice continuously
differentiable on $\cGo$, such that $f''$ (as defined on $\cG^\circ$) vanishes at infinity,
and furthermore for every $v\in V$ and all $l\in\cL$ the limit
\begin{equation*}
    \lim_{\xi\to v,\, \xi\in l\op} f''(\xi)
\end{equation*}
exists. Similarly as in the statement of lemma~\ref{lemC02} the last limit
is equal to the second order derivative $f''(v_l)$ of $f$ at $v$ in direction
of $l$.

For given data $a$, $b$, $c$ as in~\eqref{eq1i}, \eqref{eq1ii}, it will be convenient
to consider $\cH_{a,b,c}$ equivalently as being the subspace of
$\Coo$ so that for its elements $f$ at every $v\in V$ the boundary
conditions~\eqref{eq1iii} as well as the boundary condition
\begin{equation}    \label{eq2iii}
    f''(v_l) = f''(v_k),\qquad     \text{for all $l,\,k\in\cL(v)$}
\end{equation}
hold true. Relation~\eqref{eq2iii} is just another way to express that $f''$ extends
continuously from $\cG^\circ$ to $\cG$.

We consider the sets $V$, $\cE$, and $\cI$ as being ordered in some arbitrary
way. With the convention that in $\cL$ the elements of $\cE$ come first this
induces also an order relation on $\cL$.

Suppose that $f\in\Coo$. With the given ordering of $\cE$ and $\cI$ we define the
following column vectors of length $|\cE|+2|\cI|$:
\begin{align*}
    f(V)    &= \Bigl(\bigl(f_e(0),\,e\in\cE\bigr),\bigl(f_i(0),\,i\in\cI\bigr),
                                        \bigl(f_i(\rho_i)\,i\in\cI\bigr)\Bigr)^t,\\
    f'(V)   &= \Bigl(\bigl(f'_e(0),\,e\in\cE\bigr),\bigl(f'_i(0),\,i\in\cI\bigr),
                                        \bigl(-f'_i(\rho_i)\,i\in\cI\bigr)\Bigr)^t,\\
    f''(V)  &= \Bigl(\bigl(f''_e(0),\,e\in\cE\bigr),\bigl(f''_i(0),\,i\in\cI\bigr),
                                        \bigl(f''_i(\rho_i)\,i\in\cI\bigr)\Bigr)^t,
\end{align*}
where the superscript ``$t$'' indicates transposition.

We want to write the boundary conditions~\eqref{eq1iii}, \eqref{eq2iii} in a compact
way. To this end we introduce the following order relation on $V_\cL$: For $v_l$,
$v'_{l'}\in V_\cL$ we set $v_l \preceq v'_{l'}$ if and only if $v \prec v'$ or $v =
v'$ and $l\preceq l'$ (where for $V$ and $\cL$ we use the order relations introduced
above). For $f$ as above set
\begin{align*}
    \tilde f(V)   &= \bigl(f(v_l),\,v_l\in V_\cL\bigr)^t,\\
    \tilde f'(V)  &= \bigl(f'(v_l),\,v_l\in V_\cL\bigr)^t,\\
    \tilde f''(V) &= \bigl(f''(v_l),\,v_l\in V_\cL\bigr)^t.
\end{align*}
Then there exists a permutation matrix $P$ so that
\begin{equation*}
    \tilde f(V) = P f(V),\qquad \tilde f'(V) = P f'V),\qquad \tilde f''(V) = P f''(V).
\end{equation*}
In particular, $P$ is an orthogonal matrix which has in every row and in every
column exactly one entry equal to one while all other entries are zero.

For every $v\in V$ we define the following $|\cL(v)|\times|\cL(v)|$ matrices:
\begin{align*}
    \tilde A(v) &= \begin{pmatrix}
                    a_v    & 0      & 0      & \cdots & 0\\
                    0      & 0      & 0      & \cdots & 0\\
                    \vdots & \vdots & \vdots & \ddots & \vdots\\
                    0      & 0      & 0      & \cdots & 0
                 \end{pmatrix},\\[2ex]
    \tilde B(v) &= \begin{pmatrix}
                    -b_{v_{l_1}} & -b_{v_{l_2}} & -b_{v_{l_3}} & \cdots & -b_{v_{l_{|\cL(v)|}}}\\
                    0            & 0            & 0            & \cdots & 0\\
                    \vdots       & \vdots       & \vdots       & \ddots & \vdots\\
                    0            & 0            & 0            & \cdots & 0
                 \end{pmatrix},\\[2ex]
    \tilde C(v) &= \begin{pmatrix}
                    1/2\,c_v & 0      & 0      & 0      & \cdots & 0      \\
                    1        & -1     & 0      & 0      & \cdots & 0      \\
                    0        & 1      & -1     & 0      & \cdots & 0      \\
                    0        & 0      & 1      & -1     & \cdots & 0      \\
                    \vdots   & \vdots & \ddots & \ddots & \ddots & \vdots \\
                    0        & 0      & 0      & 0      & \cdots & -1
                 \end{pmatrix},
\end{align*}
where we have labeled the elements in $\cL(v)$ in such a way that in the above
defined ordering we have $l_1\prec l_2 \prec \dotsb \prec l_{|\cL(v)|}$. Observe
that $\tilde C(v)$ is invertible if and only if $c_v\ne 0$. Define block matrices
$\tilde A$, $\tilde B$, and $\tilde C$ by
\begin{equation*}
    \tilde A = \bigoplus_{v\in V} A(v),\quad \tilde B = \bigoplus_{v\in V} B(v),
            \quad \tilde C = \bigoplus_{v\in V} C(v).
\end{equation*}
Then we can write the boundary conditions~\eqref{eq1iii}, \eqref{eq2iii} simultaneously for
all vertices as
\begin{equation}    \label{bch}
    \tilde A \tilde f(V) + \tilde B \tilde f'(V) + \tilde C \tilde f''(V) = 0.
\end{equation}
Consequently the boundary conditions can equivalently be written in the form
\begin{equation}    \label{bc}
    Af(V) + Bf'(V) + Cf''(V) = 0,
\end{equation}
with
\begin{equation}
    A = P^{-1}\tilde A P,\quad B = P^{-1}\tilde B P,\quad C = P^{-1}\tilde C P.
\end{equation}

We bring in the following two matrix-valued functions on the complex plane
\begin{equation}    \label{Zpm}
    \hat Z_\pm (\gk) = A \pm \gk B + \gk^2 C,\qquad \gk\in\C.
\end{equation}

\begin{lemma}   \label{lem2iv}
There exists $R>0$ so that for all $\gk\in\C$ with $|\gk|\ge R$ the matrices $\hat
Z_\pm(\gk)$ are invertible, and there are constants $C$, $p>0$ so that
\begin{equation}    \label{invnorm}
    \|\hat Z_\pm(\gk)^{-1}\|\le C\,|\gk|^p ,\qquad |\gk|\ge R.
\end{equation}
\end{lemma}

\begin{proof}[Proof of lemma~\ref{lem2iv}]
Since we have
\begin{equation}    \label{ZP}
    \hat Z_{\pm}(\gk) = P^{-1}\bigl(\tilde A \pm\gk \tilde B + \gk^2\tilde C\bigr) P
\end{equation}
for an orthogonal matrix $P$, for the proof of the first statement it suffices
to show that there exists $R>0$ such that
\begin{equation*}
    \tilde A \pm\gk \tilde B + \gk^2\tilde C
\end{equation*}
are invertible for complex $\gk$ outside of the open ball of radius $R$. For this in turn
it suffices to show that for every vertex $v\in V$ the matrices
\begin{equation*}
\begin{split}
    \tilde A(v) &\pm\gk \tilde B(v) + \gk^2\tilde C(v)\\
        &= \begin{pmatrix}
                    a_v \pm \gk b_{v_{l_1}}+ \gk^2/2\,c_v & \pm\gk b_{v_{l_2}} & \pm\gk b_{v_{l_3}}
                             & \pm\gk b_{v_{l_4}} & \cdots & \pm\gk b_{v_{l_{|\cL(v)|}}}      \\
                    \gk^2    & -\gk^2 & 0      & 0      & \cdots & 0      \\
                    0        & \gk^2  & -\gk^2 & 0      & \cdots & 0      \\
                    0        & 0      & \gk^2  & -\gk^2 & \cdots & 0      \\
                    \vdots   & \vdots & \vdots & \ddots & \ddots & \vdots \\
                    0        & 0      & 0      & 0      & \cdots & -\gk^2
          \end{pmatrix}
\end{split}
\end{equation*}
are invertible for all $\gk\in\C$ with $|\gk|\ge R$. An elementary calculation
gives
\begin{equation*}
    \det\bigl(\tilde A(v) \pm\gk \tilde B(v) + \gk^2\tilde C(v)\bigr)
        = \Bigl(a_v \pm \gk \sum_{l\in \cL(v)}b_{v_l} + \frac{\gk^2}{2}\,c_v\Bigr)
                \bigl(-\gk^2\bigr)^{|\cL(v)|-1}.
\end{equation*}
The choices $\gk=\pm 1$ together with condition~\eqref{eq1ii} show that the polynomial
of second order in $\gk$ in the first factor on the right hand side does not vanish identically.
Therefore, it is non-zero in the exterior of an open ball with some radius $R_v>0$.
Hence, we obtain the first statement for the choice $R = \max_{v\in V} R_v$.
Moreover, from the calculation of the determinants above we also get for every $v\in V$ and
all $\gk\in\C$ with $|\gk|\ge R$ an estimate of the form
\begin{equation}    \label{detinv}
    \bigl|\det\bigl(\tilde A(v) \pm\gk \tilde B(v) + \gk^2\tilde C(v)\bigr)\bigr|^{-1}
        \le \text{const.}
\end{equation}
Thus, using the co-factor formula for
\begin{equation*}
    \bigl(\tilde A(v) \pm\gk \tilde B(v) + \gk^2\tilde C(v)\bigr)^{-1}
\end{equation*}
we find with~\eqref{detinv} the estimate
\begin{equation*}
    \bigl\|\bigl(\tilde A(v) \pm\gk \tilde B(v) + \gk^2\tilde C(v)\bigr)^{-1}\bigr\|
        \le C_v |\gk|^{p_v},\qquad |\gk|\ge R,
\end{equation*}
for some constants $C_v$, $p_v>0$. Consequently we get
\begin{equation*}
    \bigl\|\bigl(\tilde A \pm\gk \tilde B + \gk^2\tilde C\bigr)^{-1}\bigr\|
        \le C |\gk|^p,\qquad |\gk|\ge R,
\end{equation*}
for some constants $C$, $p>0$, and by~\eqref{ZP} we have proved inequality~\eqref{invnorm}.
\end{proof}

With these preparations we can enter the

\begin{proof}[Proof of lemma~\ref{lem2iii}]
Let the data $a$, $b$, $c$ be given as in~\eqref{eq1i}, \eqref{eq1ii}. We have to
show that the inclusion $\cD(A)\subset \cH_{a,b,c}$ is not strict. Assume to the
contrary that the inclusion  $\cD(A)\subset \cH_{a,b,c}$ is strict. We will derive a
contradiction.

Let $R = (R_\gl,\,\gl>0)$ be the resolvent of $A$. Then for every $\gl>0$, $R_\gl$
is a bijection from $C_0(\cG)$ onto $\cD(A)$, that is, $R_\gl^{-1}$ is a bijection from
$\cD(A)$ onto $C_0(\cG)$.  For $\gl>0$ consider the linear mapping
$H_\gl:\,f\mapsto \gl f - 1/2 f''$ from $\cH_{a,b,c}$ to $C_0(\cG)$. On $\cD(A)$
this mapping coincides with the bijection $R^{-1}_\gl$ from $\cD(A)$ onto $C_0(\cG)$.
Therefore our assumption entails that $H_\gl$ cannot be injective. Hence for any
$\gl>0$ there exists $f(\gl)\in\cH_{a,b,c}$, $f(\gl)\ne 0$, with
\begin{equation}    \label{homeq}
    H_\gl f(\gl) = \gl f(\gl) - \frac{1}{2}\,f''(\gl) = 0.
\end{equation}
We will show that $f(\gl)\in\cH_{a,b,c}$ satisfying \eqref{homeq} can only hold when
$f(\gl)=0$ on $\cG$. It will be convenient to change the
variable $\gl$ to $\gk = \sqrt{2\gl}$, and there will be no danger of confusion that
we shall simply write $f(\gk)$ for $f(\gl)$ from now on. Then the solution
of~\eqref{homeq} is necessarily of the form given by
\begin{align}
    f_e(\gk,x) &= r_e(\gk)\,e^{-\gk x}  & e&\in\cE,\,x\in \R_+,\\
    f_i(\gk,x) &= r^+_i(\gk)\,e^{\gk x}+r^-_i(\gk)\,e^{\gk(\rho_i- x)}
                                        & i&\in\cI,\,x\in [0,\rho_i],
\end{align}
and we want to show that for some $\gk > 0$, the boundary conditions~\eqref{eq1iii}
and~\eqref{eq2iii} entail that $r_e(\gk) = r^+_i(\gk) = r^-_i(\gk)=0$ for all
$e\in\cE$, $i\in\cI$. For $\gk>0$, define a column vector $r(\gk)$ of length
$|\cE|+2|\cI|$ by
\begin{equation*}
    r(\gk) = \bigl((r_e(\gk),\,e\in\cE),(r^+_i(\gk),\,i\in\cI),(r^-_i(\gk),\,i\in\cI)\bigr)^t,
\end{equation*}
and introduce the $(|\cE|+2|\cI|)\times (|\cE|+2|\cI|)$ matrices
\begin{equation*}
    X_\pm(\gk)   = \begin{pmatrix}
                    \1 & 0                & 0                \\
                    0  & \1               & \pm e^{\gk \rho} \\
                    0  & \pm e^{\gk \rho} & \1
                   \end{pmatrix}
\end{equation*}
--- appropriately modified in case that $\cE$ or $\cI$ is the empty set --- with
the $|\cI|\times|\cI|$ diagonal matrices
\begin{equation*}
    e^{\gk \rho} = \diag{e^{\gk \rho_i},\,i\in\cI\bigr}.
\end{equation*}
Then the boundary conditions~\eqref{eq1iii}, \eqref{eq2iii} for $f(\gk)$ read
\begin{equation}    \label{Zr}
    Z(\gk)r(\gk) = 0,
\end{equation}
with
\begin{equation}    \label{defZ}
    Z(\gk) = (A+\gk^2C)X_+(\gk) + \gk BX_-(\gk).
\end{equation}
Thus, if we can show that for some $\gk > 0$ the matrix $Z(\gk)$ is invertible, the
proof of the lemma is finished. Note that the matrix-valued function $Z$ is
entire in $\gk$, and therefore so is its determinant. Thus, if can show that $\gk\mapsto \det
Z(\gk)$ does not vanish identically, then it can only vanish on a discrete subset of
the complex plane, and for $\gk$ in the complement of this set $Z(\gk)$ is
invertible. Write
\begin{equation*}
    X_\pm(\gk) = \1 \pm \gd X(\gk),
\end{equation*}
with
\begin{equation*}
    \gd X(\gk) = \begin{pmatrix}
                    0  & 0                & 0            \\
                    0  & 0                & e^{\gk \rho} \\
                    0  & e^{\gk \rho}     & 0
                   \end{pmatrix},
\end{equation*}
so that we can write
\begin{equation*}
    Z(\gk) = \hat Z_+(\gk)\bigl(\1+\gd Z(\gk)\bigr),
\end{equation*}
with
\begin{equation*}
    \gd Z(\gk) = \hat Z_+(\gk)^{-1}\, \hat Z_-(\gk)\,\gd X(\gk).
\end{equation*}
Observe that in case that $\cI=\emptyset$, we obtain $\gd Z(\gk)=0$, and in this
case the invertibility of $Z(\gk)$ for all $\gk$ with $\gk\ge R$ follows from
lemma~\ref{lem2iv}. Hence we assume from now on that $\cI\ne \emptyset$.
Lemma~\ref{lem2iv} provides us with the bound
\begin{equation*}
    \bigl\|\hat Z_+(\gk)^{-1}\hat Z_-(\gk)\bigr\| \le \text{const.}\,|\gk|^{q},
\end{equation*}
for all $\gk\in\C$, $|\gk|\ge R$, and for some $q>0$. On the other hand, we get
\begin{equation*}
    \|\gd X(\gk)\| \le e^{\gk \rho_0},
\end{equation*}
for all $\gk\le 0$ where $\rho_0 = \min_{i\in\cI} \rho_i$. Therefore, there exists a
constant $R'>0$ so that for all $\gk\le -R'$  we have $\|\gd Z(\gk)\|<1$, and therefore
for such $\gk$, $Z(\gk)$ is invertible, i.e., $\det Z(\gk)\ne 0$. Hence there also
exists $\gk > 0$ so that $Z(\gk)$ is invertible, and the proof is finished.
\end{proof}

\begin{remark}\label{rem_sa_bc}
The special (and only) choice of boundary conditions in the form $c_v=0$ and
$b_{v_l}=(1-a_v)/n(v)>0$  for all $v$ also gives rise to a selfadjoint
nonpositive Laplace operator on $L^2(\mathcal{G})$, see theorem~5.1 in \cite{KoSc06}.
The associated semigroup is positivity preserving and agrees on
$C_0(\cG)\cap L^2(\cG)$ with the semigroup associated to a corresponding
Brownian motion on $\cG$. In turn
the theorem just quoted also provides examples of selfadjoint Laplace operators,
whose semigroups are positivity preserving but which are not linked to a
Brownian motion process in the above way.
\end{remark}

\section{Joining Two Metric Graphs}	\label{sect3}
For what follows it will be convenient to write $\cG = (V, \cI, \cE, \p)$ for a metric graph
$\cG$ in order to make explicit that $V$ is its set of vertices, $\cI$ its set
of internal and $\cE$ its set of external edges. $\gd$ is the map from $\cL$ into
$V\cup(V\times V)$ as defined in section~\ref{sect_1}. For simplicity we assume from now on ---
with the exception of the discussion in section~\ref{sect5} --- that the metric graphs under
consideration do not have any \emph{tadpoles}, that is, internal edges $i$ so that
$\p^-(i)=\p^+(i)$.

Throughout this section we suppose that $\cG_k=(V_k,\cI_k, \cE_k,\p_k)$, $k=1$, $2$,
are two finite metric graphs. In the following subsection we shall construct a new metric graph
$\cG=(V,\cI, \cE,\p)$ from $\cG_1$ and $\cG_2$ by connecting some of their external
edges.

It will be convenient to consider the metric graphs $\cG_1$, $\cG_2$ as
subgraphs of the metric graph $\cG_0=\cG_1\uplus\cG_2$  which is their (disjoint)
union: $\cG_0=(V_0,\cI_0,\cE_0,\p_0)$, with $V_0=V_1\cup V_2$, $\cI_0=\cI_1\cup
\cI_2$, $\cE_0=\cE_1\cup \cE_2$, and where the map $\p_0$ comprises the maps
$\p_1$, $\p_2$ in the obvious way.

\subsection{Construction of the graph \boldmath $\cG$} \label{ssect3i}
Suppose that $N$ is a natural number such that $N\le \min(|\cE_1|,|\cE_2|)$. For
$k=1$, $2$, select subsets $\cE_k'\subset \cE_k$ of edges with
$|\cE_1'|=|\cE_2'|=N$ to be joined. Let these sets be labeled as follows
\begin{equation*}
	\cE'_1 = \{e_1,\dotsc,e_N\},\quad
	\cE'_2 = \{l_1,\dotsc,l_N\}.
\end{equation*}
\begin{figure}[ht]
\begin{center}
    \includegraphics[scale=.8]{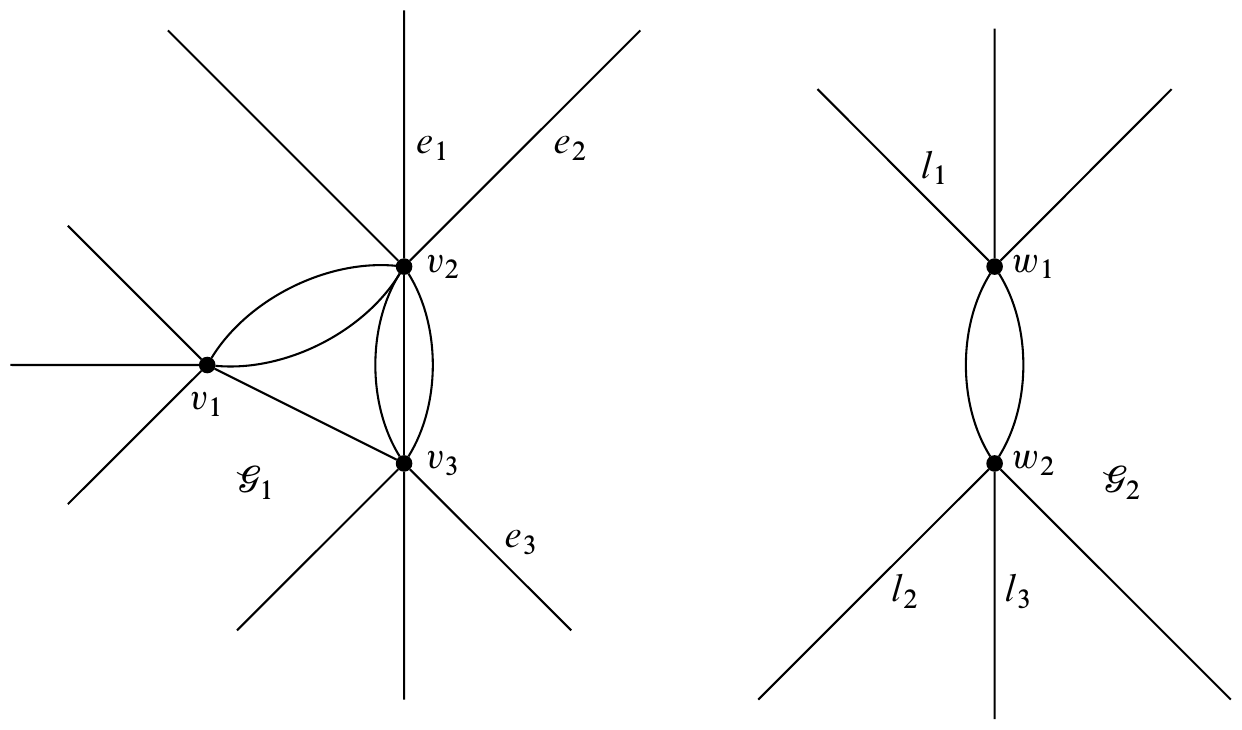}
    \caption{Two metric graphs $\cG_1$, $\cG_2$, to be joined by connecting the
             pairs of external lines
             $(e_1,l_1)$, $(e_2,l_2)$ and $(e_3,l_3)$.} \label{fig2}
\end{center}
\end{figure}

In addition we assume that we are given strictly positive numbers $b_1$, \dots, $b_N$,
which will serve as the lengths of the new internal edges, as well as $\gs_k\in\{-1,1\}$,
$k=1$, \dots, $N$, which will determine their orientations.
For every $k\in\{1,\dotsc,N\}$ we associate with the interval $[0,b_k]$ an abstract
edge $i_k$ (not in $\cI_0$) which is isomorphic to $[0,b_k]$. Set $\cI_c =
\{i_1,\dotsc,i_N\}$, and
\begin{align*}
	V   &= V_0,\\
	\cI &= \cI_0\cup\cI_c,\\
	\cE &= \cE_0\setminus (\cE'_1\cup\cE'_2).
\end{align*}
The map $\p$ is constructed in two steps: Let $\p'$ be the restriction of $\p_0$ to
$\cI_0\cup\cE_0\setminus(\cE'_1\cup\cE'_2)$. Then $\p$ is the extension of $\p'$ to
$\cI\cup\cE$, which is defined by
\begin{equation*}
	\p(i_k) = \begin{cases}
				\bigl(\p_1(e_k),\p_2(l_k)\bigr),	&\text{if $\gs_k=1$},\\
				\bigl(\p_2(l_k),\p_1(e_k)\bigr),	&\text{if $\gs_k=-1$},
			 \end{cases}
                \qquad k=1,\dotsc,N.
\end{equation*}
Figure~\ref{fig1} shows an example of a metric graph which is constructed from the
two metric graphs in figure~\ref{fig2} by joining the $N=3$ pairs of external edges
$(e_1,l_1)$, $(e_2,l_2)$ and $(e_3,l_3)$. The new internal edges $i_1$, $i_2$ and
$i_3$ have the lengths $1$, $\sqrt{2}$ and $1$ respectively (in some scale).

Conversely, let a metric graph $\cG$ be given. Associate with every vertex $v\in V$
of $\cG$ a single vertex graph $\cG(v)$ with vertex $v$ and $n(v)$ external
edges, where $n(v)$ is the number of edges incident with $v$ in $\cG$. Then it is clear
that we can reconstruct $\cG$ from the single vertex graphs $\cG(v)$, $v\in V$,
by finitely many applications of the joining procedure described above.

For the purposes below it will be convenient to introduce some additional notation.
We let $V_c\subset V$ denote the subset of vertices of $\cG$ which are connected to
each other by the new internal edges in $\cI_c$. That is, $v\in V_c$ is such that
there exists at least one $i\in\cI_c$ with $v\in\p(i)$. For notational simplicity,
here and below we also use $\p(l)$ to denote the set consisting of $\p^-(l)$ and
$\p^+(l)$ if $l\in \cI$, and of $\p(l)$ if $l\in\cE$. In the example of the
figures~\ref{fig1} and~\ref{fig2}, $V_c = \{v_2,v_3,w_1,w_2\}$.

Consider a vertex $v\in V_c$ which belongs to $\cG_1$, and let $i_k\in\cI_c$,
$k\in\{1,\dotsc,N\}$, be an internal edge connecting $v$ to $\cG_2$, i.e.,
$v\in\p(i_k)$. Then the point $\eta\in\cG_2^\circ$ with local coordinates
$(l_k,b_k)$ is called a \emph{shadow vertex} of the vertex $v$. $\sh(v)\subset \cG^0_2$
is the set of all shadow vertices of $v$. If $v\in V_c\cap\cG_2$, its set of shadow
vertices (which are points in $\cG_1^\circ$) is defined analogously. $V_s=\sh(V_c)
= \cup_{v\in V_c}\sh(v)$ is the set of all shadow vertices. If $\xi\in V_s$, then
there exists a unique $v\in V_c$ so that $\xi\in\sh(v)$. We put $\gk(\xi)=v$ and
thereby define a mapping from $V_s$ onto $V_c$. Of course, in general $\gk$ is not
injective. In figure~\ref{fig3} the shadow vertices of the example above are
depicted as small circles on the external edges, i.e., $V_s=\{\xi_1,\xi_2,\xi_3,\eta_1,
\eta_2,\eta_3\}$. For example, $\sh(v_2)=\{\eta_1,\eta_2\}$, $\sh(w_1)=\{\xi_1\}$,
and $\gk(\xi_2)=w_2$, $\gk(\eta_2)=v_2$.
\begin{figure}[ht]
\begin{center}
    \includegraphics[scale=.8]{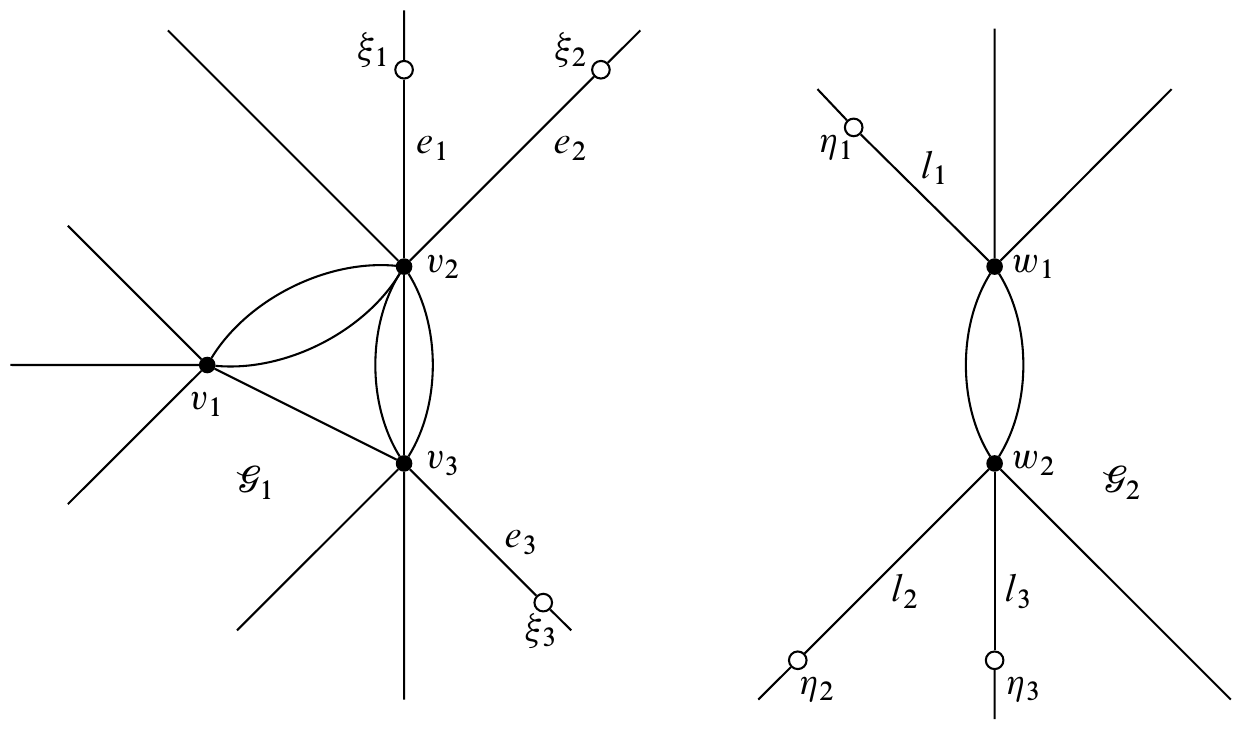}
    \caption{The graphs $\cG_1$, $\cG_2$ with the shadow vertices (small circles).}
    \label{fig3}
\end{center}
\end{figure}

\subsection{Construction of a Preliminary Version of the Brownian Motion} \label{ssect3ii}
From now we suppose that we are given a family of probability spaces
\begin{equation*}
    (\Xi^0,\cC^0,Q^0_\xi), \qquad \,\xi\in\cG_0,
\end{equation*}
and that thereon a Brownian motion with state space $\cG_0$ in the sense
of definition~\ref{def1i} is defined. This Brownian motion is
denoted by $Z^0=(Z^0(t),\,t\in\R_+)$. Actually, since $\cG_0=\cG_1\cup \cG_2$ and
$\cG_1$, $\cG_2$ are disconnected, this is the same as saying that we are given a
Brownian motion on $\cG_1$ and one on $\cG_2$. However, notationally it will be more
convenient to view this as one stochastic process. We assume, as we may, that $Z^0$
has exclusively c\`adl\`ag paths which are continuous up to the lifetime $\zeta^0$ of $Z^0$.
$\cF^0=(\cF^0_t,\,t\in\R_+)$ denotes the natural filtration of $Z^0$. The hitting
time of $V_s$ by $Z^0$ is denoted by $\tau^0$, i.e.,
\begin{equation*}
    \tau^0 = \inf\,\{t>0,\,Z^0(t) \in V_s\}.
\end{equation*}
Furthermore, we assume that $\vt=(\vt_t,\,t\in\R_+)$ is a family of shift operators
for $Z^0$ acting on $\Xi^0$.

For any topological space $(T,\cT)$ denote by $C_\gD(\R_+,T)$ the space of mappings
$\go$ from $\R_+$ into $T\cup\{\gD\}$ which are right continuous, have left limits in
$T$, are continuous up to the lifetime
\begin{equation*}
    \zeta_\go = \inf\,\{t>0,\,\go(t)= \gD\},
\end{equation*}
and which are such that $\go(t)=\gD$ implies $\go(s)=\gD$ for all $s\ge t$. In particular
and in the present context, $\go\in\CD(\R_+,\cG_0)$ is either continuous from $\R_+$ into
$\cG_0$ or it has a jump from $\cG_1$ or $\cG_2$ to $\gD$, but there can be no jump
from $\cG_1$ to $\cG_2$ or vice versa.

We shall make use of some special versions of the process $Z^0$, which we introduce
now. For every $v\in V_c$, $Z^1_v=(Z^1_v(t),\,t\in\R_+)$ denotes a Brownian motion
on $\cG_0$ defined on another probability space $(\Xi^1_v,\cC^1_v,\mu^1_v)$ such that
under $\mu^1_v$, $Z^1_v$ is equivalent to $Z^0$ under $Q^0_v$. We
suppose that $Z^1_v$ exclusively has paths which start in $v$ and which belong to
$\CD(\R_+,\cG_0)$. (For example, one can use a standard path space construction to
obtain such a version from $(\Xi^0,\cC^0, Q^0_v, Z^0)$.) The hitting time of $V_s$ by
$Z^1_v$ is denoted by $\tau^1_v$, its lifetime by $\zeta^1_v$.

The idea to define the preliminary version $Y=(Y(t),\,t\in\R_+)$ of the Brownian
motion on $\cG$ is to construct its paths as follows. Let $\xi\in\cG$ be a given
starting point. $\cG$ (viewed as a set) has the following decomposition (cf.\
figure~\ref{fig4}):
\begin{equation*}
	\cG =  \hat\cG_1\uplus\hat\cG_2,
\end{equation*}
with
\begin{align*}
	\hat\cG_1 &= \cG_1\setminus \bigl(e^\circ_1\cup\dotsb\cup e^\circ_N\bigr),\\
    \hat\cG_2 &= \bigl(\cG_2\setminus \bigl(l^\circ_1\cup\dotsb\cup l^\circ_N\bigr)\bigr)
                        \cup \bigl(i_1^\circ\cup\dotsc\cup i_N^\circ\bigr).
\end{align*}
\begin{figure}[ht]
\begin{center}
    \includegraphics[scale=.8]{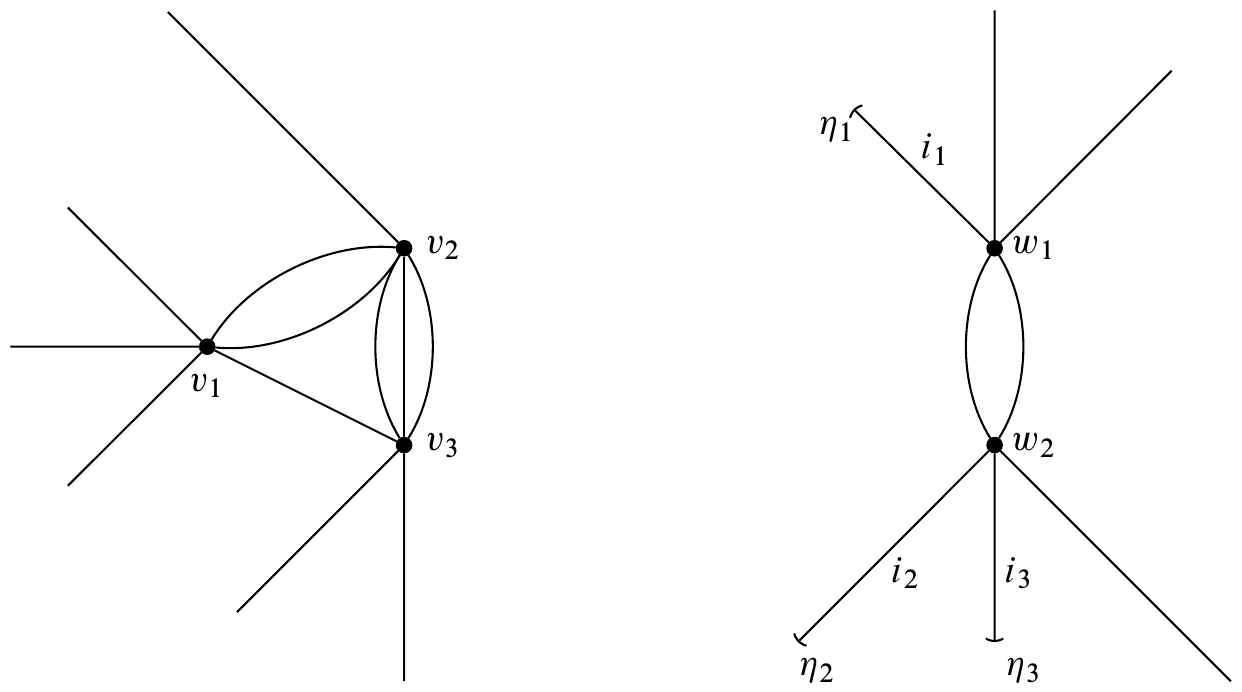}
    \caption{The starting points of $Y$.} \label{fig4}
\end{center}
\end{figure}
Thus we may consider $\xi$ instead as a point in $\hat\cG_1\uplus\hat\cG_2\subset
\cG_0$.

We pause here for the following remark: Of course, the convention we make that
all new open inner edges $i^\circ_1,\dotsc,i^\circ_N$ are attached to $\hat\cG_2$
is somewhat arbitrary. Just as well any subset of them could have been attached
to $\hat\cG_1$ instead. Even though different conventions lead to processes
with different paths, the main result of this section, theorem~\ref{thm3xiii},
remains unchanged. It follows that all resulting processes are equivalent to each
other.

Let $Y$ start as $Z^0$ in $\xi\in\hat\cG_1\uplus\hat\cG_2$, and consider one trajectory.
(In order to avoid any confusion, let us point out that even though $\xi\in\hat\cG_k$,
$k=1$, $2$, the process $Z^0$ moves in $\cG_k$.) If this
trajectory reaches the cemetery point $\gD$ before hitting the set $V_s$ of shadow
vertices, it is the complete trajectory of $Y$ and it stays forever at the cemetery.
If the trajectory hits a shadow vertex $\eta\in V_s$ before its lifetime expires,
this piece of the trajectory of $Y$ ends at the hitting time $\tau^0$. Set
$v=\gk(\eta)$, and let the trajectory of $Y$ continue with an (independent) trajectory
of $Z^1_v$ until its lifetime expires or it hits a shadow vertex, and so on.
Figure~\ref{fig5} explains the idea.
\begin{figure}[ht]
\begin{center}
    \includegraphics[scale=.8]{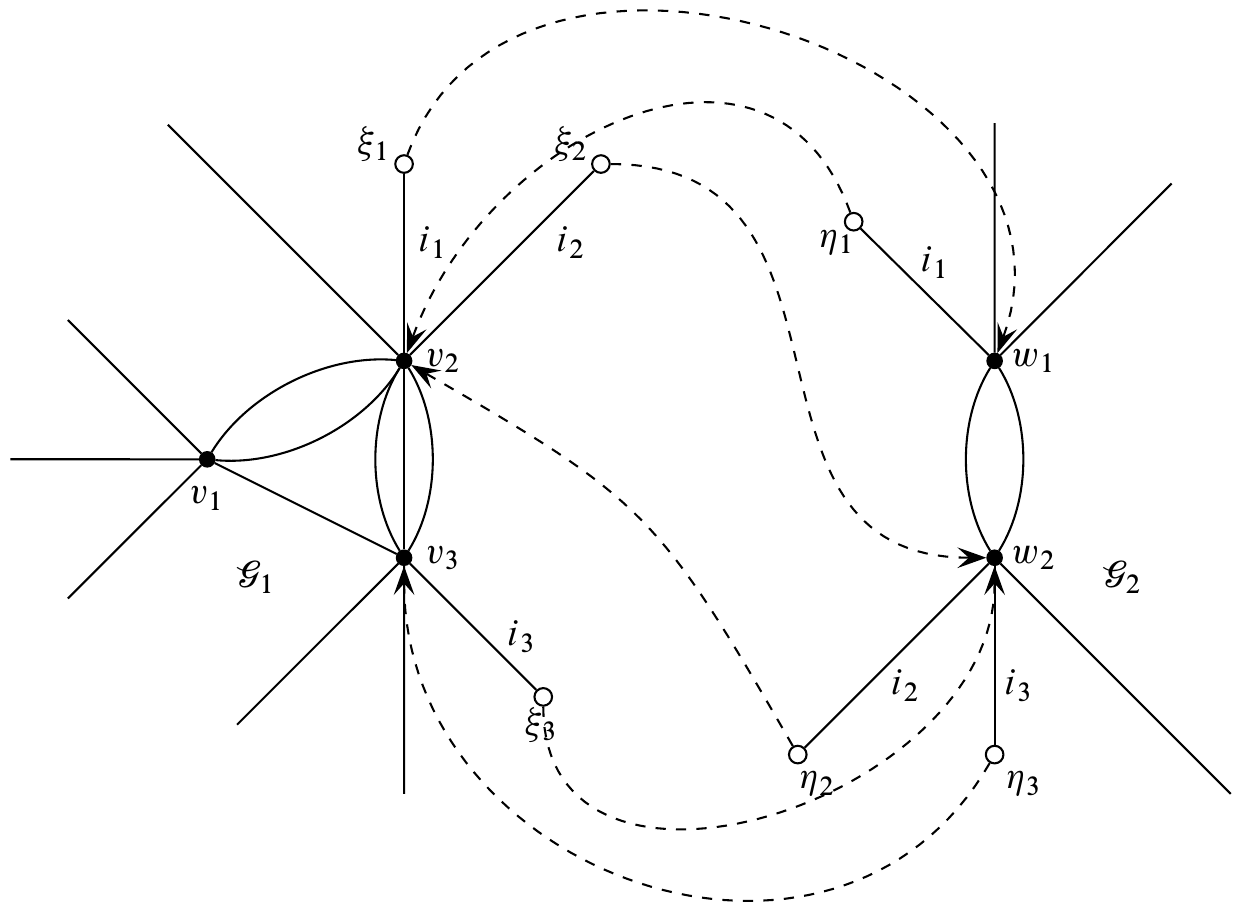}
    \caption{The construction of the process $Y$.} \label{fig5}
\end{center}
\end{figure}

The construction described above is formalized in the following way. Define
\begin{equation*}
    \Xi^1 = \BigCart_{v\in V_c} \Xi^1_v,\quad
    \cC^1 = \bigotimes_{v\in V_c} \cC^1_v,\quad
    Q^1   = \bigotimes_{v\in V_c} \mu^1_v,
\end{equation*}
and view each of the stochastic processes $Z^1_v$, $v\in V_c$, as well as the random
variables $\tau^1_v$, $\zeta^1_v$, as defined on this product space. Let
\begin{equation*}
    \bigl(\Xi^n,\cC^n,Q^n,Z^n,\tau^n,\zeta^n\bigr),\qquad n\in\N,\,n\ge 2,
\end{equation*}
be an independent sequence of copies of
\begin{equation*}
    \bigl(\Xi^1,\cC^1,Q^1,Z^1,\tau^1,\zeta^1\bigr),
\end{equation*}
where $Z^1=(Z^1_v,\,v\in V_c)$ and similarly for $\tau^1$, $\zeta^1$. Next set
\begin{equation*}
    \Xi   = \BigCart_{n=0}^\infty \Xi^n,\quad
    \cC   = \bigotimes_{n=0}^\infty \cC^n,\quad
    Q_\xi = Q^0_\xi \otimes\Bigl(\bigotimes_{n=1}^\infty Q^n\Bigr),\  \xi\in\cG.
\end{equation*}

The procedure sketched above of pasting together pieces of the trajectories of the
various processes $Z^n_v$ is controlled by a Markov chain $(K_n,\,n\in\N)$ which
moves at random times $(S_n,\,n\in\N)$ in the state space $V_c\cup\{\gD\}$. We
set out to construct this chain $\bigl((S_n,K_n),\,n\in\N\bigr)$. Define $S_1=\tau^0$. On
$\{S_1=+\infty\}$, i.e., in the case when $\zeta^0<\tau^0$, set $K_1 = \gD$.
Otherwise define
\begin{equation*}
    K_1 = \gk\bigl(Z^0(\tau^0)\bigr).
\end{equation*}
Observe that since all processes considered have right continuous paths, they are all
measurable stochastic processes, and therefore the evaluation of their time argument at
a random time yields a well-defined random variable. Set $S_2=+\infty$ on
$\{S_1=+\infty\}$, while
\begin{equation*}
    S_2 = S_1 + \tau^1_{K_1}
\end{equation*}
on $\{S_1<+\infty\}$. On $\{S_2=+\infty\}$ put $K_2=\gD$, and on its complement
\begin{equation*}
    K_2 = \gk\bigl(Z^1_{K_1}(\tau^1_{K_1})\bigr).
\end{equation*}
These construction steps are iterated in the obvious way: The
sequence
\begin{equation*}
    \bigl((S_n,K_n),\,n\in\N\bigr)
\end{equation*}
is inductively defined by $S_n=+\infty$ and $K_n=\gD$ on $\{S_{n-1}=+\infty\}$,
while
\begin{align*}
	S_n &= S_{n-1} + \tau^{n-1}_{K_{n-1}},\\[1ex]
	K_n &= \gk\bigl(Z^{n-1}_{K_{n-1}}(\tau^{n-1}_{K_{n-1}})\bigr)
\end{align*}
on $\{S_{n-1}<+\infty\}$.

Note that by construction $K_n = \gD$, $n\in\N$, if and only if $S_n=+\infty$,
and in that case $K_{n'}=\gD$, $S_{n'}=+\infty$ for all $n'\ge n$. Thus
$(+\infty,\gD)$ is a cemetery state for the chain $((S_n,K_n),\,n\in\N)$.

For example with a Borel--Cantelli argument it is not hard to see (cf.\
also~\cite{BMMG0}) that there exists a set $\Xi'\in\cC$ so that for all $\xi\in\cG$,
$Q_\xi(\Xi')=0$, and for all $\go\in\Xi\setminus\Xi'$  the sequence $(S_n(\go),\,n\in\N)$
increases to $+\infty$ in such a way that for all $n\in\N$, $S_n(\go)<S_{n+1}(\go)$ holds
when $S_n(\go)<+\infty$.

Now we are ready to construct $Y=(Y(t),\in\R_+)$. Let $\xi\in\cG =
\hat\cG_1\uplus\hat\cG_2$ be a given starting point, and suppose that $t\in\R_+$ is
given. On $\Xi'$ set $Y(t)=\gD$. On $\Xi\setminus\Xi'$ there is a unique $n\in\N_0$
so that $t\in [S_n, S_{n+1})$, with the convention $S_0=0$. If $t\in [0,S_1)$,
define $Y(t) = Z^0(t)$. If $t\in [S_n, S_{n+1})$ for $n\in\N$, then necessarily
$S_n$ is finite, so that $K_n\in V_c$, and we define
\begin{equation*}
	Y(t) = Z^n_{K_n}(t-S_n).
\end{equation*}
In addition, we make the convention $Y(+\infty)=\gD$. The natural
filtration generated by $Y$ will be denoted by $\cF^Y=(\cF^Y_t,\,t\in\R_+)$.

It follows from the construction of $Y$ that $\gD$ is a cemetery state for $Y$.
Indeed, suppose that $\go\in\Xi\setminus \Xi'$, and that the trajectory
$Y(\,\cdot\,,\go)$ reaches the point $\gD$ at a finite time $\zeta^Y(\go)$. This
implies that there is an $n\in\N_0$ such that $\zeta^Y(\go)\in
[S_n(\go),S_{n+1}(\go))$. Then $S_n(\go)$ is finite, and therefore $K_n(\go)\in V_c$
so that $Y(\,\cdot\,,\go)$ is equal to $Z^n_{K_n}(\,\cdot\,-S_n(\go),\go)$ on the
interval $[S_n(\go),S_{n+1}(\go))$, and this trajectory reaches $\gD$ before hitting
a shadow vertex. Hence $\tau^n_{K_n}=+\infty$ which entails that
$S_{n+1}(\go)=+\infty$. Consequently after $S_n(\go)$ there are no finite
crossover times for this trajectory, and therefore $Y(\,\cdot\,,\go)$ stays
at~$\gD$ forever. Furthermore note that the left limit $Y(\zeta^Y(\go)-,\go)$
at $\zeta^Y(\go)$ belongs to $V_0$.

In terms of the stochastic process $Y$ the random times $S_n$, $n\in\N$, have the
following description. Suppose that $Y$ starts in $\xi\in\cG$. Then $S_1$ is the
hitting time of $V_c$. But if $\xi\in V_c$, then actually it is the hitting time of
$V_c\setminus\{\xi\}$, because it hits a vertex in $V_c$ which corresponds to the
first hitting of a shadow vertex, i.e., a point in $\cG_0$ different from $\xi$,
by $Z^0$. In particular, $S_1>0$. Similarly, $S_n$ is the hitting time of
$V_c\setminus\{K_{n-1}\}$ by $Y$ after time $S_{n-1}$. In appendix~\ref{appA} it
is shown that for every $n\in\N$, $S_n$ is a stopping time with respect to $\cF^Y$.

It follows from its construction that $Y$ is a normal process,
that is, for every $\xi\in\cG$, $Q_\xi(Y(t=0)=\xi)=1$. Furthermore, all paths of $Y$ belong
to $\CD(\R_+,\cG)$. Let $S_V$ be the hitting time of the set of vertices $V$ of
$\cG$ by $Y$. Then $S_V \le S_1$, because $V_c\subset V$ and therefore we find that
$Y(\,\cdot\,\land S_V)$ is pathwise equal to $Z^0(\,\cdot\,\land S^0_V)$, where
$S^0_V$ denotes the hitting time of $V$ by $Z^0$. Suppose that the starting point
$\xi$ belongs to $l^\circ$, $l\in\cI\cup\cE$, and $l$ is isomorphic to the interval
$I$. Then by definition of $Z^0$ (cf.\ definition~\ref{def1i}), the stopped
process $Z^0(\,\cdot\,\land S^0_V)$ is equivalent to a standard Brownian motion on the
interval $I$ with absorption at the endpoint(s) of $I$. Hence the same is true for
$Y$: $Y(\,\cdot\,\land S_V)$ is equivalent to a standard Brownian motion on $I$
with absorption at the endpoint(s) of $I$.

\subsection{Markov property of \boldmath $Y$} \label{ssect3iii}
For any measurable space $(M,\cM)$, $B(M)$ denotes the space of bounded, measurable
functions on $M$. Every $f\in B(\cG^n)$, $n\in\N$, is extended to
$(\cG\cup\{\gD\})^n$ by $f(\xi_1,\dotsc,\xi_n)=0$,
$(\xi_1,\dotsc,\xi_n)\in(\cG\cup\{\gD\})^n$, whenever there is an index
$k\in\{1,\dotsc, n\}$ so that $\xi_k=\gD$. In this subsection we shall prove the
following

\begin{proposition} \label{prop3i}
$Y$ has the simple Markov property: For all $f\in B(\cG)$, $s$, $t\in\R_+$,
$\xi\in\cG$,
\begin{equation}    \label{eq3i}
    E_\xi\bigl(f\bigl(Y(s+t)\bigr)\cond \cF^Y_s\bigr)
        = E_{Y(s)}\bigl(f\bigl(Y(t)\bigr)\bigr)
\end{equation}
holds true $Q_\xi$--a.s.\ on $\{Y(s)\ne\gD\}$.
\end{proposition}

The proof of proposition~\ref{prop3i} is somewhat technical and lengthy.
Therefore it will be broken up into a sequence of lemmas.

For every $n\in\N$ the probability space $(\Xi,\cC,Q_\xi)$, $\xi\in\cG$, underlying
the construction of the process $Y$ may be written as the product of the probability
spaces $(\Xil, \cCl, \Ql_\xi)$ and $(\Xiu, \cCu, \Qu)$ with
\begin{align*}
	\Xil &= \BigCart_{j=0}^{n-1} \Xi^j,  &
				\Xiu &= \BigCart_{j=n}^{\infty} \Xi^j,\\
	\cCl &= \bigotimes_{j=0}^{n-1} \cC^j, &
				\cCu &= \bigotimes_{j=n}^{\infty} \cC^j,\\
	\Ql_\xi &= Q_\xi^0\otimes\Bigl(\bigotimes_{j=1}^{n-1} Q^j\Bigr), &
				\Qu &= \bigotimes_{j=n}^{\infty} Q^j.
\end{align*}
Introduce a family $\cB=(\cB_n,\,n\in\N_0)$ of sub--$\gs$--algebras of $\cC$ by
setting
\begin{equation*}
    \cB_n = \cC^{\le n-1}\times \Xi^{\ge n}.
\end{equation*}
Obviously, the family $\cB$ forms a filtration. Furthermore, from the construction
of $K_n$ and $S_n$ it is easy to see that the chain $((S_n, K_n),\,n\in\N)$ is
adapted to $\cB$.

First we study the chain $((S_n,K_n),\,n\in\N)$ in more detail. Recall our
convention that $S_0=0$. We set $\cB_0=\{\emptyset,\Xi\}$, and under the law
$Q_v$, $v\in V_c$, we put $K_0=v$. $g\in\bRV$ is extended to $\Rbp\times
\bigl(V_c\cup\{\gD\}\bigr)$ by $g(+\infty,\,\cdot\,)=g(\,\cdot\,,\gD)
=g(+\infty,\gD)=0$. For $g\in\bRV$, $n\in\N_0$, define
\begin{equation}    \label{eq3ii}
    (U_n g)(s,v) = E_{v}\bigl(g(s+S_n, K_n)\bigr),\qquad s\in\R_+,\,v\in V_c.
\end{equation}
Note that $U_0=\text{id}$, and that for every $g\in\bRV$ and all $n\in\N$,
$U_n g\in\bRV$. In particular, the convention mentioned above applies to $U_n g$,
too.
\goodbreak

\begin{lemma}   \label{lem3ii}
{\ }
\begin{enum_a}
    \item For all $m$, $n\in\N$, $m\le n$,  $\xi\in\cG$, $s\ge 0$, and every
        $g\in\bRV$ the following formula holds true $Q_\xi$--a.s.\
            \begin{equation}    \label{eq3iii}
                E_\xi\bigl(g(s+S_n, K_n) \cond \cB_m)\bigr)
                    = (U_{n-m}g)(s+S_m,K_m).
            \end{equation}
    \item $(U_n,\,n\in\N_0)$ forms a semigroup of linear maps on $\bRV$. In
        particular, for all $(s,v)\in\R_+\times V_c$ under $Q_{v}$ the chain
        $\bigl((s+S_n, K_n),\,n\in\N_0\bigr)$ is a homogeneous Markov chain with
        transition kernel
            \begin{equation*}
                P\bigl((s,v), A\bigr)
                    = Q_{v}\bigl((s+S_1,K_1)\in A\bigr),\qquad A\in\cB(\R_+\times V_c).
            \end{equation*}
\end{enum_a}
\end{lemma}

\begin{proof}
For $m=n$ formula~\eqref{eq3iii} is trivial. Consider the case when $n\ge 2$, $m=n-1$.
Let $\gL\in\cB_{n-1}$, $v\in V_c$, and put $\gL_v = \gL\cap\{K_{n-1}=v\}\in\cB_{n-1}$.
From the construction of $S_n$ and $K_n$
\begin{align*}
    E_\xi\bigl(&g(s+S_n,K_n); \gL_v\bigr)\\[1ex]
        &\hspace{-1.5em}= \int_{\Xil[n-2]} 1_{\gL_v}
            \Bigl(\int_{\Xiu[n-1]}g\bigl(s+S_{n-1}+\tau^{n-1}_v,
                    \gk\bigl(Z^{n-1}_v(\tau^{n-1}_v)\bigr)\bigr)\,d\Qu[n-1]\Bigr)\,d\Ql[n-2]_\xi\\[1ex]
        &\hspace{-1.5em}= \int_{\Xil[n-2]} 1_{\gL_v}
            \Bigl(\int_{\Xi^0}g\bigl(s+u+\tau^0,
                    \gk\bigl(Z^0(\tau^0)\bigr)\bigr)\,dQ^0_v\Bigr)\Eval_{u=S_{n-1}}\,d\Ql[n-2]_\xi\\[1ex]
        &\hspace{-1.5em}=E_\xi\Bigl(E_v\bigr(g(s+u+S_1,K_1)\bigr)\eval_{u=S_{n-1}};\gL_v\Bigr)\\[1ex]
        &\hspace{-1.5em}=E_\xi\Bigl(E_{K_{n-1}}\bigr(g(s+u+S_1,K_1)\bigr)\eval_{u=S_{n-1}};\gL_v\Bigr)\\[1ex]
        &\hspace{-1.5em}=E_\xi\Bigl( (U_1 g)(s+S_{n-1},K_{n-1});\gL_v\Bigr),
\end{align*}
where in the last step we used definition~\eqref{eq3ii}. If in the preceding calculation
we replace the event $\{K_{n-1}=v\}$ by $\{K_{n-1}=\gD\}$, we get zero on both sides because
$K_{n-1}=\gD$ implies $K_n=\gD$ (see subsection~\ref{ssect3ii}). Thus summation over
$v\in V_c$ gives
\begin{equation*}
    E_\xi\bigl(g(s+S_n,K_n);\gL\bigr)
        = E_\xi\bigl((U_1 g)(s+S_{n-1}, K_{n-1});\gL\bigr),
\end{equation*}
and equation~\eqref{eq3iii} is proved for the case where $n\ge 2$ and $m=n-1$. As a
consequence we get
\begin{align*}
    (U_n g)(s,v)
        &= E_v\bigl(E_v\bigl(g(s+S_n,K_n)\cond \cB_{n-1}\bigr)\bigr)\\
        &= E_v\bigl((U_1 g)(s+S_{n-1},K_{n-1})\bigr)\\
        &= \bigl(U_{n-1}\comp U_1 g\bigr)(s,v).
\end{align*}
Now the general semigroup relation $U_{n+m}=U_n\comp U_m$, $m$, $n\in\N_0$, follows
by an application of Fubini's theorem.

Finally we show formula~\eqref{eq3iii} in the general case:
\begin{align*}
    E_\xi\bigl(g(s+S_n,K_n)&\cond \cB_m\bigr)\\
        &= E_\xi\bigl(E_\xi\bigl(g(s+S_n,K_n)\cond \cB_{n-1}\bigr)\cond \cB_m\bigr)\\
        &= E_\xi\bigl((U_1 g)(s+S_{n-1},K_{n-1})\cond\cB_m\bigr)\\
        &=\quad \dots \quad=\\
        &= E_\xi\bigl((U_1\comp\dotsb\comp U_1 g)(s+S_m,K_m)\cond\cB_m\bigr)\\
\end{align*}
where the multiple composition in the last expression involves $n-m$ operators $U_1$.
The semigroup property of $(U_n,\,n\in\N_0)$ implies formula~\eqref{eq3iii}.
\end{proof}

It will be useful to introduce some additional notation. For $r\in\N$, let
$\Rrp$ denote the set of all increasingly ordered $r$--tuples with entries in
$\R_+$. If $u\in\Rrp$ and $s\in\R$ we set $u+s=(u_1+s,\dotsc, u_r+s)\in\R^r$. $u<s$ means
that $u_i<s$ for all $i=1$, \dots, $r$ or equivalently $u_r<s$. The relations
$u>s$, $u\le s$, and $u\ge s$ are defined analogously. In particular,
when $s\le u$, then $u-s\in\Rrp$. For $r$, $q\in\N$, and $u\in\Rrp$, $w\in\Rrp[q]$
with $u\le w_1$, define
\begin{equation*}
    (u,w) = (u_1,\dotsc, u_r,w_1,\dotsc,w_q)\in\Rrp[r+q].
\end{equation*}
Furthermore, $Y(u)$ stands for $(Y(u_1),\dotsc, Y(u_r))$, and similarly for $Z^n_v(u)$,
$n\in\N_0$, $v\in V_c$.

In the sequel we shall consider random variables $W_m(h,g,u)$ of the following form
\begin{equation}    \label{eq3iv}
    W_m(h,g,u) = h\bigl(Y(u)\bigr)\,\chi_m(u)\,g(S_{m+1},K_{m+1}),
\end{equation}
where $m\in\N$, $h$ belongs to $\bcBG$,  $r\in \N$, $g$ to $\bRV$, and $u\in\Rrp$. Here
we have set
\begin{equation*}
    \chi_m(u) = 1_{\{S_m\le u <S_{m+1}\}}.
\end{equation*}
For $s\ge 0$ with $s\le u$ define
\begin{equation}    \label{eq3v}
    W_{m,s}(h,g,u)
        = h\bigl(Y(u-s)\bigr)\,\chi_m(u-s)\,g(s+S_{m+1},K_{m+1}),
\end{equation}
so that $W_{m,s=0}(h,g,u)=W_m(h,g,u)$. Moreover set
\begin{equation}    \label{eq3vi}
    R_m(h,g,u)(s,v) = E_v\bigl(W_{m,s}(h,g,u)\bigr),\qquad s\in\R_+,\,s\le u,\,v\in V_c.
\end{equation}
For the following it will be convenient to let $W_{m,s}(h,g,u)$ and $R_m(h,g,u)(s,v)$,
$v\in V_c$, be defined for all $s\in\R_+$. To this end we make the convention that
$Y(t)=\gD$ for all $t<0$. Then by $W_{m,s}(h,1,u)=W_m(h,1,u-s)$ the following formula
\begin{equation}    \label{eq3via}
    R_m(h,1,u)(s+t,v) = R_m(h,1,u-s)(t,v)
\end{equation}
holds for all $m\in\N$, $h\in B(\cG^r)$, $r\in\N$, $u\in\Rrp$, $s$, $t\in\R_+$,
$v\in V_c$.

Suppose that $r$, $q\in\N$, and that $h\in\bcBG$, $f\in\bcBG[q]$. Then $h\otimes f$
denotes the function in $B(\cG^{r+q})$ given by
\begin{equation}    \label{eq3vib}
\begin{split}
    h\otimes f&(\eta_1,\dotsc,\eta_{r+q})\\
        &= h(\eta_1,\dotsc,\eta_r)\,f(\eta_{r+1},\dotsc,\eta_{r+q}),\qquad
                    (\eta_1,\dotsc,\eta_{r+q})\in\cG^{r+q}.
\end{split}
\end{equation}

\begin{lemma}   \label{lem3iii}
Suppose that $r\in\N$, $u\in\Rrp$, and that $h\in\bcBG$.
\begin{enum_a}
    \item If $q\in\N$, $w\in \Rrp[q]$ with $u_r\le w$, and
            $f\in\bcBG[q]$, then
            \begin{subequations}     \label{eq3vii}
            \begin{equation}    \label{eq3viia}
                R_0\bigl(h\otimes f, 1, (u,w)\bigr)= R_0\big(M(f,w,u_r)h, 1,u\bigr)
            \end{equation}
            holds true, where $M(f,w,s)h\in B(\cG^r)$ is given by
            \begin{equation}    \label{eq3viib}
                \bigl(M(f,w,s)h\bigr)(\eta) = h(\eta)\,E_{\eta_r}\bigl(W_{0,s}(f,1,w)\bigr),
                                \quad \eta\in\cG^r,\,0\le s\le w.
            \end{equation}
            \end{subequations}
    \item If $g\in\bRV$, then
            \begin{subequations}     \label{eq3viii}
            \begin{equation}    \label{eq3viiia}
                R_0\bigl(h,g,u\bigr) = R_0\bigl(N(g,u_r)h,1,u\bigr)
            \end{equation}
            holds, where $N(g,s)h\in B(\cG^r)$ is given by
            \begin{equation}    \label{eq3viiib}
                \bigl(N(g,s) h\bigr)(\eta)= h(\eta)\,E_{\eta_r}\bigl(g(s+S_1,K_1)\bigr),
                                            \qquad \eta\in\cG^r,\,s\ge 0.
            \end{equation}
            \end{subequations}
\end{enum_a}
\end{lemma}

\begin{proof}
Both statements follow from the Markov property of the Brownian motion $Z^0$ on
$\cG_0$ underlying the construction of $Y$. We only prove statement~(b), the
proof of~(a) is similar and therefore omitted. Using the definition of $R_0$ and the
construction of $Y$, we compute for $s\in\R_+$, $v\in V_c$, as follows:
\begin{align*}
    R_0\bigl(h,g,u\bigr)(s,v)
        &= E_v\bigl(h\bigl(Y(u-s)\bigr)\,\chi_0(u-s)\,g(s+S_1,K_1)\bigr)\\
        &= E_v\bigl(h\bigl(Z^0(u-s)\bigr)\,1_{\{0\le u-s <\tau^0\}}\,
                g\bigl(s+\tau^0,\gk(Z^0(\tau^0))\bigr)\bigr).
\end{align*}
Recall that $\cF^0$ denotes the natural filtration of $Z^0$, and $\vt$ is a family
of shift operators for $Z^0$. It follows from the definition of the stopping time
$\tau^0$ and the path properties of $Z^0$, that on $\{\tau^0\ge u_r-s\}$ the
relation $\tau^0 = u_r-s + \tau^0\comp\vt_{u_r-s}$ holds true. Moreover, it is
easy to check that on this event we have $Z^0(\tau^0)=Z^0(\tau^0)\comp\vt_{u_r-s}$.
Therefore
\begin{align*}
    R_0\bigl(h,g,u\bigr)(s,v)
        &= E_v\Bigl(h\bigl(Z^0(u-s)\bigr)\,1_{\{0\le u-s <\tau^0\}}\\
        &\hspace{2em} \times E_v\bigl(g(u_r+\tau^0,\gk(Z^0(\tau^0)))
                            \comp\vt_{u_r-s}\cond \cF^0_{u_r-s}\bigr)\Bigr)\\
        &= E_v\Bigl(h\bigl(Z^0(u-s)\bigr)\,1_{\{0\le u-s <\tau^0\}}\\
        &\hspace{2em} \times E_{Z^0(u_r-s)}\bigl(g(u_r+\tau^0,\gk(Z^0(\tau^0)))\bigr)\Bigr)\\
        &= E_v\Bigl(h\bigl(Y(u-s)\bigr)\,\chi_0(u-s)\,
            E_{Y(u_r-s)}\bigl(g(u_r+S_1,K_1)\bigr)\Bigr)\\
        &= R_0\bigl(N(g,u_r)h,1,u\bigr)(s,v),
\end{align*}
and the proof is concluded.
\end{proof}

\begin{lemma}   \label{lem3iv}
For all $m$, $r\in\N$, $h\in\bcBG$, $g\in\bRV$, $u\in\Rrp$, $\xi\in\cG$, the formula
\begin{equation}    \label{eq3ix}
    E_\xi\bigl(W_m(h,g,u)\cond \cB_m\bigr)
        = R_0(h,g,u)(S_m,K_m)
\end{equation}
holds $Q_\xi$--a.s.
\end{lemma}

\begin{proof}
Observe that both side of equation~\eqref{eq3ix} vanish on the set $\{K_m=\gD\}$.
Let $\gL\in\cB_m$, $v\in V_c$, and set $\gL_v=\gL\cap\{K_m=v\}\in\cB_m$. Then
\begin{align*}
    E_\xi\bigl(&W_m(h,g,u);\gL_v\bigr)\\
        &= \int_{\Xil[m-1]} 1_{\gL_v}\, \Bigl(\int_{\Xiu[m]} h\bigl(Y(u)\bigr)\,
                    1_{\{S_m\le u<S_{m+1}\}}\\
        &\hspace{7em} \times g(S_{m+1},K_{m+1})\,d\Qu[m]\Bigr)\,d\Ql[m-1]_\xi\\
        &= \int_{\Xil[m-1]} 1_{\gL_v}\, \Bigl(\int_{\Xi^m} h\bigl(Z^m_v(u-s)\bigr)\,
                    1_{\{0\le u-s < \tau^m_v\}}\\
        &\hspace{7em} \times g\bigl(s+\tau^m_v,\gk(Z^m_v(\tau^m_v))\bigr)\,dQ^m\Bigr)
                \Eval_{s=S_m}\,d\Ql[m-1]_\xi\\
        &= \int_{\Xil[m-1]} 1_{\gL_v}\, \Bigl(\int_{\Xi^0} h\bigl(Z^0(u-s)\bigr)\,
                    1_{\{0\le u-s < \tau^0\}}\\
        &\hspace{7em} \times g\bigl(s+\tau^0,\gk(Z^0(\tau^0))\bigr)\,dQ^0_v\Bigr)
                \Eval_{s=S_m}\,d\Ql[m-1]_\xi\\
        &= E_\xi\Bigl(E_v\bigl(h\bigl(Y(u-s)\bigr)\,\chi_0(u-s)\,g(s+S_1,K_1)\bigr)
            \eval_{s=S_m};\,\gL_v\Bigr)\\
        &= E_\xi\Bigl(E_{K_m}\bigl(h\bigl(Y(u-s)\bigr)\,\chi_0(u-s)\,g(s+S_1,K_1)\bigr)
            \eval_{s=S_m};\,\gL_v\Bigr)\\
        &= E_\xi\bigl(R_0(h,g,u)(S_m,K_m);\,\gL_v\bigr).
\end{align*}
Summation over $v\in V_c$ finishes the proof.
\end{proof}

\begin{lemma}   \label{lem3v}
For all $m$, $r\in\N$, $u\in\Rrp$, $h\in\bcBG$, $g\in\bRV$,
\begin{equation}    \label{eq3x}
    R_m(h,g,u)(s,v) = \bigl(U_m R_0(h,g,u)\bigr)(s,v),\qquad s\in\R_+,\,v\in V_c,
\end{equation}
holds.
\end{lemma}

\begin{proof}
By definition of $U_m$
\begin{equation*}
    \bigl(U_m R_0(h,g,u)\bigr)(s,v)
        = E_v\bigl(R_0(h,g,u)(s+S_m,K_m)\bigr).
\end{equation*}
With formula~\eqref{eq3via} and lemma~\ref{lem3iv} we find
\begin{align*}
    \bigl(U_m R_0(h,g,u)\bigr)(s,v)
        &= E_v\bigl(R_0(h,g(s+\,\cdot\,),u-s)(S_m,K_m)\bigr)\\
        &= E_v\bigl(E_v\bigl(W_m(h,g(s+\,\cdot\,),u-s)\cond \cB_m\bigr)\bigr)\\
        &= E_v\bigl(W_m(h,g(s+\,\cdot\,),u-s)\bigr)\\
        &= E_v\bigl(W_{m,s}(h,g,u)\bigr)\\
        &= R_m(h,g,u)(s,v).\qedhere
\end{align*}
\end{proof}

\begin{corollary}   \label{cor2vi}
For all $m$, $n$, $r\in\N$, $u\in\Rrp$, $h\in\bcBG$, $g\in\bRV$,
\begin{equation}    \label{eq3xi}
    U_n R_m(h,g,u) = R_{n+m}(h,g,u).
\end{equation}
is valid.
\end{corollary}

\begin{proof} By lemma~\ref{lem3v} and lemma~\ref{lem3ii}, statement~(b), we obtain
\begin{equation*}
    U_n R_m(h,g,u) = U_n\comp U_m R_0(h,g,u) = U_{n+m} R_0(h,g,u) = R_{n+m}(h,g,u).\qedhere
\end{equation*}
\end{proof}

\begin{lemma}   \label{lem3vii}
For all $m$, $n\in\N$, $m\le n$, $r\in\N$, $u\in\Rrp$, $h\in\bcBG$, $g\in\bRV$, $\xi\in\cG$,
the following formula holds true:
\begin{equation}    \label{eq3xii}
    E_\xi\bigl(W_n(h,g,u)\cond \cB_m\bigr) = R_{n-m}(h,g,u)(S_m,K_m).
\end{equation}
\end{lemma}

\begin{proof}
Apply lemma~\ref{lem3iv} to compute as follows
\begin{align*}
    E_\xi\bigl(W_n(h,g,u)\cond \cB_m\bigr)
        &= E_\xi\bigl(E_\xi\bigl(W_n(h,g,u)\cond \cB_n\bigr)\cond \cB_m\bigr)\\
        &= E_\xi\bigl(R_0(h,g,u)(S_n,K_n)\cond \cB_m\bigr)\\
        &= \bigl(U_{n-m} R_0(h,g,u)\bigr)(S_m,K_m),
\end{align*}
where we used lemma~\ref{lem3ii}, formula~\eqref{eq3iii}, in the last step. An application of
lemma~\ref{lem3v} concludes the proof.
\end{proof}

With these preparations, we are ready for the

\begin{proof}[Proof of Proposition~\ref{prop3i}]
Assume that $f\in\bcB(\cG)$, $s$, $t\in\R_+$, and that $\xi\in\cG$. Since
$(S_m,\,m\in\N)$ $Q_\xi$--a.s.\ strictly increases to $+\infty$, and since $S_m$,
$m\in\N$, is an $\cF^Y$--stopping time (cf., lemma~\ref{lemA} in appendix~\ref{appA}),
it suffices to prove that equation~\eqref{eq3i} holds $Q_\xi$--a.s.\ for every
$m\in\N_0$ on $\{S_m\le s <S_{m+1}, Y(s)\ne \gD\}\in\cF^Y_s$. We fix an arbitrary
$m\in\N_0$. Clearly, the family of random variables of the form
\begin{equation*}
	g\bigl(Y(w)\bigr)\,1_{\{0\le w <S_m\}}\,W_m(h,1,u)
\end{equation*}
with $r$, $q\in\N$, $u\in\Rrp$, $u_r=s$, $w\in\Rrp[q]$, $h\in\bcBG$, and
$g\in\bcBG[q]$, generates the $\gs$--algebra $\cF_s^Y\cap\{S_m\le s<S_{m+1}\}$.
Therefore it is sufficient to show that
\begin{equation}   \label{eq3xiii}
\begin{split}
    E_\xi\bigl(g\bigl(Y(w)\bigr)\,&1_{\{0\le w <S_m\}}\,
				W_m(h,1,u)\,f\bigl(Y(s+t)\bigr)\bigr)\\[1ex]
        &= E_\xi\bigl(g\bigl(Y(w)\bigr)\,1_{\{0\le w <S_m\}}\,W_m(h,1,u)\,
            E_{Y(s)}\bigl(f(Y(t))\bigr)\bigr),
\end{split}
\end{equation}
holds for all $r$, $q\in\N$, $u\in\Rrp$ with $u_r=s$, $w\in\Rrp[q]$, $h\in\bcBG$,
$g\in\bcBG[q]$ and $f\in B(\cG)$. (Since the random variables under the
expectation signs of both sides of equation~\eqref{eq3xiii} vanish on the
set $\{Y(s)=\gD\}$, we can henceforth safely ignore the condition $Y(s)\ne \gD$.)
Expand the left hand side of equation~\eqref{eq3xiii} as follows:
\begin{equation}   \label{eq3xiv}
\begin{split}
    E_\xi\bigl(g\bigl(&Y(w)\bigr)\,1_{\{0\le w <S_m\}}\,W_m(h,1,u)\,f\bigl(Y(s+t)\bigr)\bigr)\\[1ex]
        &= \sum_{n=m}^\infty E_\xi\bigl(g\bigl(Y(w)\bigr)\,1_{\{0\le w <S_m\}}\,W_m(h,1,u)\,
                W_n(f,1,s+t)\bigr).
\end{split}
\end{equation}
Consider the summand with $n=m$, which is of the form
\begin{align*}
    E_\xi\bigl(&g\bigl(Y(w)\bigr)\,1_{\{0\le w < S_m\}}\,W_m(h\otimes f,1,(u,s+t))\bigr)\\
        &= E_\xi\bigl(g\bigl(Y(w)\bigr)\,1_{\{0\le w < S_m\}}\,
            E_\xi\bigl(W_m(h\otimes f,1,(u,s+t))\cond \cB_m\bigr)\bigr)\\
        &= E_\xi\bigl(g\bigl(Y(w)\bigr)\,1_{\{0\le w < S_m\}}
                    \,R_0(h\otimes f, 1, (u,s+t))(S_m,K_m)\bigr),
\end{align*}
where we made use of formula~\eqref{eq3ix}. Now we apply statement~(a) of lemma~\ref{lem3iii}
with the choice $q=1$ which yields (recall that $u_r=s$)
\begin{align*}
    E_\xi\bigl(&g\bigl(Y(w)\bigr)\,1_{\{0\le w < S_m\}}\,W_m(h\otimes f,1,(u,s+t))\bigr)\\
        &= E_\xi\bigl(g\bigl(Y(w)\bigr)\,1_{\{0\le w < S_m\}}
                    \,R_0(M(f,s+t,s)h, 1, u)(S_m,K_m)\bigr)\\
        &= E_\xi\bigl(g\bigl(Y(w)\bigr)\,1_{\{0\le w < S_m\}}
                    \,E_\xi\bigl(W_m(M(f,s+t,s)h, 1, u)\cond\cB_m\bigr)\bigr)\\
        &= E_\xi\bigl(g\bigl(Y(w)\bigr)\,1_{\{0\le w < S_m\}}\,
                    W_m(M(f,s+t,s)h, 1, u)\bigr),
\end{align*}
where we used formula~\eqref{eq3ix} again. Combining~\eqref{eq3iv} with~\eqref{eq3viib}
in $W_m(M(f,s+t,s)h, 1, u)$ we thus have shown
\begin{equation}    \label{eq3xv}
\begin{split}
    E_\xi\bigl(&g\bigl(Y(w)\bigr)\,1_{\{0\le w < S_m\}}\,W_m(h,1,u)\,W_m(f,1,s+t)\bigr)\\
        &= E_\xi\bigl(g\bigl(Y(w)\bigr)\,1_{\{0\le w < S_m\}}\,W_m(h,1,u)\,
                E_{Y(s)}\bigl(f(Y(t))\,1_{\{0<t<S_1\}}\bigr)\bigr).
\end{split}
\end{equation}
Next consider a generic summand with $n>m$ on the right hand side of~\eqref{eq3xiv}. Then
\begin{align*}
    E_\xi\bigl(&g\bigl(Y(w)\bigr)\,1_{\{0\le w<S_m\}}\,W_m(h,1,u)\,W_n(f,1,s+t)\bigr)\\
        &= E_\xi\bigl(g\bigl(Y(w)\bigr)\,1_{\{0\le w<S_m\}}\,W_m(h,1,u)\,
                E_\xi\bigl(W_n(f,1,s+t)\cond \cB_{m+1}\bigr)\bigr)\\
        &= E_\xi\bigl(g\bigl(Y(w)\bigr)\,1_{\{0\le w<S_m\}}\,W_m(h,1,u)\,
                    R_{n-m-1}(f,1,s+t)(S_{m+1},K_{m+1})\bigr)\\
        &= E_\xi\bigl(g\bigl(Y(w)\bigr)\,1_{\{0\le w<S_m\}}\,W_m(h,R_{n-m-1}(f,1,s+t),u)\bigr)
\end{align*}
where we used lemma~\ref{lem3vii} in the second step. Conditioning on $\cB_m$ gives
\begin{align*}
    E_\xi\bigl(&g\bigl(Y(w)\bigr)\,1_{\{0\le w<S_m\}}\,W_m(h,1,u)\,W_n(f,1,s+t)\bigr)\\
        &= E_\xi\bigl(g\bigl(Y(w)\bigr)\,1_{\{0\le w<S_m\}}\,
                E_\xi\bigl(W_m(h,R_{n-m-1}(f,1,s+t),u)\cond \cB_m\bigr)\bigr)\\
        &= E_\xi\bigl(g\bigl(Y(w)\bigr)\,1_{\{0\le w<S_m\}}\,
                R_0(h,R_{n-m-1}(f,1,s+t),u)(S_m,K_m)\bigr)\\
        &= E_\xi\bigl(g\bigl(Y(w)\bigr)\,1_{\{0\le w<S_m\}}\,
                R_0\bigl(N(R_{n-m-1}(f,1,s+t),s)h,1,u\bigr)(S_m,K_m)\bigr).
\end{align*}
Here we used lemmas~\ref{lem3iv}, \ref{lem3iii}.b, and in the last step also $u_r=s$.
Applying lemma~\ref{lem3iv}, we get
\begin{align*}
    E_\xi\bigl(g\bigl(Y(w)\bigr)\,&1_{\{0\le w<S_m\}}\,W_m(h,1,u)\,W_n(f,1,s+t)\bigr)\\
	   &=  E_\xi\Bigl(g\bigl(Y(w)\bigr)\,1_{\{0\le w<S_m\}}\\
	   &\hspace{4em}\times E_{\xi}\bigl(W_m(N(R_{n-m-1}(f,1,s+t),s)h,1,u)
				\cond \cB_m\bigr)\Bigr)\\
        &= E_\xi\bigl(g\bigl(Y(w)\bigr)\,1_{\{0\le w<S_m\}}\,
                W_m\bigl(N(R_{n-m-1}(f,1,s+t),s)h,1,u\bigr)\bigr)\\
        &= E_\xi\bigl(g\bigl(Y(w)\bigr)\,1_{\{0\le w<S_m\}}\,
                h\bigl(Y(u)\bigr)\,\chi_m(u)\\
        &\hspace{4em}\times E_{Y(s)}\bigl(R_{n-m-1}(f,1,s+t)(s+S_1,K_1)\bigr)\bigr).
\end{align*}
For $\eta\in\cG$ relation~\eqref{eq3via} yields
\begin{align*}
    E_\eta\bigl(R_{n-m-1}(f,1,s+t)(s+S_1,K_1)\bigr)
        &= E_\eta\bigl(R_{n-m-1}(f,1,t)(S_1,K_1)\bigr)\\
        &= E_\eta\bigl(E_\eta\bigl(W_{n-m}(f,1,t)\cond\cB_1\bigr)\bigr)\\
        &= E_\eta\bigl(W_{n-m}(f,1,t)\bigr),
\end{align*}
with another application of formula~\eqref{eq3xii}. With the choice $\eta=Y(s)$ this
relation therefore gives
\begin{equation}    \label{eq3xvi}
\begin{split}
    E_\xi\bigl(g\bigl(Y(w)\bigr)\,&1_{\{0\le w<S_m\}}\,W_m(h,1,u)\,W_n(f,1,s+t)\bigr)\\
        &= E_\xi\Bigl(g\bigl(Y(w)\bigr)\,1_{\{0\le w<S_m\}}\,h\bigl(Y(u)\bigr)\,\chi_m(u)\\
        &\hspace{7em}\times E_{Y(s)}\bigl(f(Y(t))\,1_{\{S_{n-m}\le t<S_{n-m+1}\}}\bigr)\Bigr).
\end{split}
\end{equation}
Formulae~\eqref{eq3xv} and~\eqref{eq3xvi} entail
\begin{align*}
    \sum_{n=m}^\infty &E_\xi\Bigl(g\bigl(Y(w)\bigr)\,1_{\{0\le w <S_m\}}\,W_m(h,1,u)\,
                W_n(f,1,s+t)\Bigr)\\[1ex]
        &= E_\xi\Bigl(g\bigl(Y(w)\bigr)\,1_{\{0\le w < S_m\}}\,W_m(h,1,u)\,
                E_{Y(s)}\bigl(f(Y(t))\,1_{\{0\le t<S_1\}}\bigr)\Bigr)\\
        &\hspace{4em}+ \sum_{n=m+1}^\infty E_\xi\Bigl(g\bigl(Y(w)\bigr)\,1_{\{0\le w<S_m\}}\,
                h\bigl(Y(u)\bigr)\,\chi_m(u)\\
        &\hspace{12em}\times E_{Y(s)}\bigl(f(Y(t))\,1_{\{S_{n-m}\le t<S_{n-m+1}\}}\bigr)\Bigr)\\[1ex]
        &= E_\xi\Bigl(g\bigl(Y(w)\bigr)\,1_{\{0\le w < S_m\}}\,W_m(h,1,u)\,
                E_{Y(s)}\bigl(f(Y(t))\bigr)\Bigr),
\end{align*}
which proves equation~\eqref{eq3xiii}.
\end{proof}

\subsection{A Brownian motion on \boldmath $\cG$ and its generator}\label{ssect3iv}

The stochastic process $Y$ and its underlying probability family $(\Xi,\cC,Q)$, $Q =
(Q_\xi,\,\xi\in\cG)$, are not very convenient to work with. Therefore we introduce
another version in this subsection. As the underlying sample space $\gO$ we choose
the path space $\CD(\R_+,\cG)$ of $Y$ endowed with the $\gs$--algebra $\cA$ generated
by the cylinder sets of $\CD(\R_+,\cG)$. Obviously, $Y$ is a measurable mapping from
$(\Xi,\cC)$ into $(\gO,\cA)$. For $\xi\in\cG$ let $P_\xi$ denote the image measure
of $Q_\xi$ under $Y$. Set $P=(P_\xi,\,\xi\in\cG)$. Moreover, let the canonical
coordinate process on $(\gO,\cA)$ be denoted by  $X=(X_t,\,t\in\R_+)$. Clearly, $X$
is a version of $Y$. We set $X_{+\infty}=\gD$, and denote the natural filtration of
$X$ by $\cF=(\cF_t,\,t\in\R_+)$. As usual $\cF_\infty$ stands for
$\gs(\cF_t,\,t\in\R_+)$. Whenever it is notationally more convenient we shall also
write $X(t)$ for $X_t$, $t\in\R_+$.

Let $H_V$ denote the hitting time of the set $V$ of vertices of $\cG$ by $X$:
\begin{equation*}
    H_V = \inf\{t > 0,\,X_t\in V\}.
\end{equation*}
Suppose that $X$ starts in $\xi\in l^\circ$, $l\in\cI\cup\cE$, and that $l$ is
isomorphic to the interval $I$. Then it follows directly from the discussion at the
end of subsection~\ref{ssect3ii} that the stopped process $X(\,\cdot\,\land H_V)$ is
equivalent to a standard Brownian motion on $I$ with absorption in the endpoint(s)
of $I$. The necessary path properties of $X$ being obvious, we therefore find that
$X$ satisfies all defining properties of a Brownian motion on $\cG$ (cf.\
definition~\ref{def1i}), except that we still have to
prove its strong Markov property. This will be done next.

Let $\theta = (\theta_t,\,t\in\R_+)$ denote the natural family of shift operators on
$\gO$: $\theta_t(\go) = \go(t+\,\cdot\,)$ for $\go\in\gO$. Thus in particular
$\theta$ is a family of shift operators for $X$. Since the simple Markov property is
a property of the finite dimensional distributions of a stochastic process, and the
finite dimensional distributions of $X$ and $Y$ coincide, it immediately follows
from proposition~\ref{prop3i} that $X$ is a Markov process. Then standard monotone
class arguments (e.g., \cite{KaSh91, ReYo91}) give the Markov property in the familiar
general form:

\begin{proposition}   \label{prop3viii}
Assume that $\xi\in\cG$, $t\in\R_+$, and that $W$ is an $\cF_\infty$--measurable,
positive or integrable random variable on $(\gO,\cA, P)$. Then
\begin{equation}    \label{eq3xvii}
    E_\xi\bigl(W\comp \theta_t \cond \cF_t\bigr) = E_{X_t}\bigl(W\bigr),
\end{equation}
holds true $P_\xi$--a.s.\ on $\{X_t\ne \gD\}$.
\end{proposition}

A routine argument based on the path properties of $X$ (similar to, but much easier
than the one used in the proof of lemma~\ref{lemA} in appendix~\ref{appA}) shows
that $H_V$ is an $\cF$--stopping time. We have the following

\begin{lemma}   \label{lem3ix}
$X$ has the strong Markov property with respect to the hitting time $H_V$. That is,
for all $\xi\in\cG$, $t\in\R_+$, $f\in\bcB(\cG)$,
\begin{equation}    \label{eq3xviii}
    E_\xi\bigl(f(X_{t+H_V})\cond \cF_{H_V}\bigr)
        = E_{X_{H_V}}\bigl(f(X_t)\bigr)
\end{equation}
holds true $P_\xi$--a.s.
\end{lemma}

\begin{proof}
To begin with, observe that since $\gO=\CD(\cG)$ and $X$ is the canonical coordinate
process, there is a natural family $\ga=(\ga_t,\,t\in\R_+)$ of stopping operators for
$X$, namely $\ga_t(\go) = \go(\,\cdot\,\land t)$, $t\in\R_+$. Therefore we get that
$\cF_T = \gs(X_{s\land T},\,s\in\R_+)$ for any stopping time $T$ relative to $\cF$.
Indeed, one can show this along the same lines used to prove Galmarino's
theorem (e.g., \cite[p.~458]{Ba91}, \cite[p.~86]{ItMc74}, \cite[p.~43~ff]{Kn81},
\cite[p.~45]{ReYo91}). Therefore it is sufficient to prove that for all $n\in\N$,
$s_1$, \dots, $s_n\in\R_+$, $t\in\R_+$, $\xi\in\cG$, and all $g\in\bcB(\cG^n)$,
$f\in\bcB(\cG)$, the following formula
\begin{equation}    \label{eq3xix}
\begin{split}
    E_\xi\bigl(g(X(s_1&\land H_V),\dotsc,X(s_1\land H_V))\,f(X(t+ H_V))\bigr)\\
        &= E_\xi\bigl(g(X(s_1\land H_V),\dotsc,X(s_1\land H_V))\,
                     E_{X(H_V)}\bigl(f(X(t))\bigr)\bigr)
\end{split}
\end{equation}
holds. Recall that $S_V$ denotes the hitting time of $V$ by $Y$. Since $P_\xi$ is
the image of $Q_\xi$ under $Y$, and since $S_V = H_V\comp Y$,
equation~\eqref{eq3xix} is equivalent to
\begin{equation}    \label{eq3xx}
    E_\xi\bigl(G\,f(Y(t+S_V))\bigr)
        = E_\xi\bigl(G\,E_{Y(S_V)}\bigl(f(Y(t))\bigr)\bigr),
\end{equation}
where we have set
\begin{equation*}
    G = g\bigl(Y(s_1\land S_V),\dotsc, Y(s_n\land S_V)\bigr).
\end{equation*}
Recall that $S_1$ denotes the hitting time of $V_c\subset V$ by $Y$, so that $S_V\le
S_1$. $Y$ is progressively measurable relative to $\cF^Y$ which entails that $G$ is
measurable with respect to $\cF^Y_{S_V}\subset \cF^Y_{S_1}\subset \cB_1$ (see also
the corresponding argument in the proof of lemma~\ref{lemA}). Using the notation of
subsection~\ref{ssect3iii} we write
\begin{equation}    \label{eq3xxi}
\begin{split}
    E_\xi\big(G\, &f(Y(t+S_V))\bigr)\\
        &= E_\xi\big(G\, f(Y(t+S_V);\,S_V\le t+S_V < S_1)\bigr)\\
        &\hspace{4em} + \sum_{n=1}^\infty E_\xi\bigl(G\,E_\xi\bigl(f(Y(t+u))\,
                    \chi_n(t+u)\cond \cB_1\bigr)\eval_{u=S_V}\bigr).
\end{split}
\end{equation}
For the last equality --- similarly as in the proof of lemma~\ref{lem3ii} ---
we made use of the product structure of the probability space $(\Xi,\cC,Q_\xi)$,
and the fact that $S_V\le S_1$, which entails that $S_V$ only depends on the
variable $\go^0\in \Xi^0$. By lemma~\ref{lem3vii} and formula~\eqref{eq3via} we
get for $u\le S_1$, $n\in\N$,
\begin{align*}
    E_\xi\bigl(f(Y(t+u))\,\chi_n(t+u)\cond \cB_1\bigr)
        &= R_{n-1}(f,1,t+u)(S_1,K_1)\\
        &= R_{n-1}(f,1,t)(S_1-u,K_1).
\end{align*}
Then
\begin{align*}
    E_\xi\bigl(G\, R_{n-1}&(f,1,t)(S_1-S_V,K_1)\bigr)\\
        &= E_\xi\bigl(G\,R_{n-1}(f,1,t)(0,K_1);\,S_V=S_1\bigr)\\
        &\hspace{4em} +E_\xi\bigl(G\,R_{n-1}(f,1,t)(S_1-S_V,K_1);\,S_V<S_1\bigr)\\
        &= E_\xi\bigl(G\,E_{Y(S_V)}\bigl(f(Y(t))\,\chi_{n-1}(t)\bigr);\,S_V=S_1\bigr)\\
        &\hspace{4em} +E_\xi\bigl(G\,R_{n-1}(f,1,t)(S_1-S_V,K_1);\,S_V<S_1\bigr),
\end{align*}
because on $S_V=S_1$, $Y(S_V)=Y(S_1)=K_1$. The second term on the right hand side of
the last equality only involves the random variables $Y(s_i\land S_V)$, $S_1$,
$S_V$, and $K_1$. They are all defined in terms of the strong Markov process $Z^0$
underlying the construction of $Y$ (cf.\ section~\ref{ssect3ii}). Moreover, on the
event $\{S_V < S_1\}$ we get from the definition of $S_1$ as the hitting time of
$V_c$ that $S_1 = S_V + S_1\comp \vt_{S_V}$. Also, on $\{S_V < S_1\}$,
$K_1 = K_1\circ\vt_{S_V}$ holds true.
On the other hand $G$ is measurable with respect to $\cF^0_{S_V}$, where $\cF^0$ is
the natural filtration of $Z^0$. Thus the strong Markov property of $Z^0$ gives
\begin{align*}
    E_\xi\bigl(G\,&R_{n-1}(f,1,t)(S_1-S_V,K_1);\,S_V<S_1\bigr)\\
        &= E_\xi\bigl(G\,E_{Y(S_V)}\bigl(R_{n-1}(f,1,t)(S_1,K_1)\bigr);\,S_V<S_1\bigr)\\
        &= E_\xi\bigl(G\,E_{Y(S_V)}\bigl(f(Y(t))\,\chi_n(t)\bigr);\,S_V<S_1\bigr).
\end{align*}
In the last step we used for $\eta\in\cG$,
\begin{align*}
    E_\eta\bigl(R_{n-1}(f,1,t)(S_1,K_1)\bigr)
        &= E_\eta\bigl(E_\eta\bigl(W_n(f,1,t)\cond \cB_1\bigr)\bigr)\\
        &= E_\eta\bigl(f(Y(t))\,\chi_n(t)\bigr),
\end{align*}
with another application of lemma~\ref{lem3vii}, and then we made the choice $\eta=
Y(S_V)$. Similarly, for the first term on the right hand side of equation~\eqref{eq3xxi}
we can use the strong Markov property of $Z^0$ (together with $\{t+S_V<S_1\}
= \{t<S_1\comp\vt_{S_V}\}\cap\{S_V<S_1\}$) to show that
\begin{equation*}
\begin{split}
    E_\xi\bigl(G\,f(Y(t+S_V));\,&S_V\le t+S_V < S_1\bigr)\\
        &= E_\xi\bigl(G\,E_{Y(S_V)}\bigl(f(Y(t);\,t<S_1)\bigr);\,S_V<S_1\bigr).
\end{split}
\end{equation*}
Inserting these results into the right hand side of formula~\eqref{eq3xxi}, we find
\begin{align*}
    E_\xi\bigl(&G\,f(Y(t+S_V))\bigr)\\
        &= E_\xi\bigl(G\,E_{Y(S_V)}\bigl(f(Y(t));\,t<S_1\bigr);\,S_V<S_1\bigr)\\
        &\hspace{4em} +\sum_{n=1}^\infty E_\xi\bigl(G\,E_{Y(S_V)}\bigl(f(Y(t))\,
                            \chi_n(t)\bigr);\,S_V<S_1\bigr)\\
        &\hspace{4em} +\sum_{n=0}^\infty E_\xi\bigl(G\,E_{Y(S_V)}\bigl(f(Y(t))\,
                            \chi_n(t)\bigr);\,S_V=S_1\bigr)\\
        &= E_\xi\bigl(G\,E_{Y(S_V)}\bigl(f(Y(t))\bigr)\bigr).
\end{align*}
Thus equation~\eqref{eq3xx} holds and lemma~\ref{lem3ix} is proved.
\end{proof}

\begin{proposition}   \label{prop3x}
$X$ is a Feller process.
\end{proposition}

\begin{proof}
It is well-known, that it is sufficient to prove~(i) that the resolvent of $X$
preserves $C_0(\cG)$, and~(ii) that for all $f\in C_0(\cG)$, $\xi\in\cG$,
$E_\xi\bigl(f(X_t)\bigr)$ converges to $f(\xi)$ as $t$ decreases to $0$. (A complete
proof can be found in appendix~\ref{app_FSR}.) Statement~(ii) immediately
follows by an application of the dominated convergence theorem and the fact that $X$
is a normal process with right continuous paths.

To prove statement~(i) consider the resolvent $R = (R_\gl,\,\gl>0)$ of $X$, and let
$\gl>0$. Since $X$ is strongly Markovian with respect to the hitting time $H_V$ of
the set of vertices $V$ (lemma~\ref{lem3ix}), we get for $\xi\in\cG$,
$f\in\bcB(\cG)$ the first passage time formula (e.g., \cite{Ra56} or
\cite{ItMc74})
\begin{equation}    \label{eq3xxii}
    (R_\gl f)(\xi)
        = (R^D_\gl f)(\xi) + E_\xi\bigl(e^{-\gl H_V}\,(R_\gl f)(X_{H_V})\bigr),
\end{equation}
where $R^D$ is the Dirichlet resolvent. That is, $R^D$ is the resolvent of the
process $X$ with killing at the moment of reaching a vertex of $\cG$. Recall the
equivalence of the stopped process $X(\,\cdot\,\land H_V)$ with the Brownian motion
with absorption on the corresponding interval $I$ stated at the beginning of this
subsection. Then we can give explicit expressions for all entities appearing in the
first passage time formula~\eqref{eq3xxii}. Using the well-known formulae for
Brownian motions on the real line (see, e.g., \cite[p.~73ff]{DyJu69},
\cite[p.~29f]{ItMc74}) we find for $\xi$, $\eta\in\cG^\circ$,
\begin{subequations}    \label{eq3xxiii}
\begin{equation}    \label{eq3xxiiia}
    r^D_\gl(\xi,\eta)
        = \sum_{i\in \cI} r^D_{\gl,i}(\xi,\eta)\,1_{\{\xi,\eta\in i\}}
            + \sum_{e\in\cE} r^D_{\gl,e}(\xi,\eta)\,1_{\{\xi,\eta\in e\}},
\end{equation}
with
\begin{equation}    \label{eq3xxiiib}
    r^D_{\gl,i}(\xi,\eta)
        = \frac{1}{\sgl}\,\sum_{k\in\Z}\Bigl(e^{-\sgl|x-y+2ka_i|}-e^{-\sgl|x+y+2ka_i|}\Bigr),
\end{equation}
where in local coordinates $\xi=(i,x)$, $\eta=(i,y)$, $x$, $y\in (0,a_i)$. In the case of
an external edge $e$, we get
\begin{equation}    \label{eq3xxiiic}
    r^D_{\gl,e}(\xi,\eta) = \frac{1}{\sgl}\,\Bigl(e^{-\sgl|x-y|}-e^{-\sgl(x+y)}\Bigr),
\end{equation}
\end{subequations}
with $\xi=(e,x)$, $\eta=(e,y)$, $x$, $y\in (0+\infty)$. Remark that both kernels
vanish whenever $\xi$ or $\eta$ converge from the interior of any edge to a vertex
to which the edge is incident.

Consider the second term on the right hand side of equation~\eqref{eq3xxii}. Suppose that
$\xi\in i^\circ$, $i\in \cI$, and that $i$ is isomorphic to $[0,a_i]$. Assume furthermore
that $v_1$, $v_2$ are the vertices in $V$ to which $i$ is incident, and that under this
isomorphism $v_1$ corresponds to $0$, while $v_2$ corresponds to $a_i$. Then we get
\begin{equation*}
\begin{split}
    E_\xi\bigl(e^{-\gl H_V}\,(R_\gl f)(X_{H_V})\bigr)
        = E_\xi\bigl(&e^{-\gl H_{v_1}};\,H_{v_1}<H_{v_2}\bigr)\,(R_\gl f)(v_1)\\
            &+  E_\xi\bigl(e^{-\gl H_{v_2}};\,H_{v_2}<H_{v_1}\bigr)\,(R_\gl f)(v_2),
\end{split}
\end{equation*}
because $X$ has paths which are continuous up to the lifetime of $X$, and $X$ cannot
be killed before reaching a vertex. Here $H_{v_k}$, $k=1$, $2$, denotes the hitting
time of the vertex $v_k$.  The expectation values in the last line are those of a
standard Brownian motion and they are well-known, too (see, e.g.,
\cite[p.~73ff]{DyJu69}, \cite[p.~29f]{ItMc74}). Thus for $\xi = (i,x)$, $x\in
[0,a_i]$,
\begin{equation}    \label{eq3xxiv}
\begin{split}
    E_\xi\bigl(&e^{-\gl H_V}\,(R_\gl f)(X_{H_V})\bigr)\\
        &= \frac{\sinh\bigl(\sgl (a_i-x)\bigr)}{\sinh\bigl(\sgl a_i\bigr)}\,(R_\gl f)(v_1)
            + \frac{\sinh\bigl(\sgl x\bigr)}{\sinh\bigl(\sgl a_i\bigr)}\,(R_\gl f)(v_2).
\end{split}
\end{equation}
Similarly, for $\xi\in e^\circ$ with local coordinates $\xi=(e,x)$, $x\in (0,+\infty)$
we find
\begin{equation}    \label{eq3xxv}
    E_\xi\bigl(e^{-\gl H_V}\,(R_\gl f)(X_{H_V})\bigr)
        = e^{-\sgl x}\,(R_\gl f)(v),
\end{equation}
where $v$ is the vertex to which $e$ is incident.

With the formulae~\eqref{eq3xxii}--\eqref{eq3xxv} it is straightforward to check
that $R_\gl$ maps $C_0(\cG)$ into itself, and the proof is complete.
\end{proof}

Since $X$ has right continuous paths, standard results (see, e.g,
\cite[Theorem~III.3.1]{ReYo91}, or \cite[Theorem~III.15.3]{Wi79}) provide the

\begin{corollary}   \label{cor3xi}
$X$ is strongly Markovian.
\end{corollary}

Thus we have also proved the

\begin{corollary}   \label{cor3xii}
$X$ is a Brownian motion on $\cG$ in the sense of definition~\ref{def1i}.
\end{corollary}

It remains to calculate the domain of the generator of $X$, i.e., the boundary
conditions at the vertices. Let $v\in V$, and assume that $X$ starts in $v$. Then by
construction of $X$ and $Y$, $X$ is equivalent to $Z^0$ with start in $v$ up to its
first hitting of a shadow vertex. That is, $X$ is equivalent to $Z^0$ up to the
first time $X$ hits a vertex different from $v$. It follows that if $v$ is absorbing
for $Z^0$, then it is so for $X$, and if $v$ is an exponential holding point with
jump to $\gD$, then it is also so for $X$ with the same exponential rate. In
particular, $v$ is a trap for $X$ if and only if it is a trap for $Z^0$. If $v$ is
not a trap, then we can use Dynkin's formula~\cite[p.~140, ff.]{Dy65a} to calculate
the boundary condition implemented by $X$. Clearly, this gives the same boundary
conditions as for $Z^0$, because Dynkin's formula only involves an arbitrary small
neighborhood of the vertex. (See also the corresponding arguments in
section~\ref{sect2}.) Thus we have proved the following

\begin{theorem} \label{thm3xiii}
$X$ is a Brownian motion on $\cG$ whose generator has a domain characterized by
the same boundary conditions as the generator of $Z^0$.
\end{theorem}

\section{Proof of Theorem~\ref{thm1ii}}	\label{sect4}
Suppose that $\cG=(V,\cI,\cE,\p)$ is a metric graph without tadpoles. Let data $a$,
$b$, $c$ as in~\eqref{eq1i} be given which satisfy equation~\eqref{eq1ii}.

With every $v\in V$ we associate a single vertex graph $\cG(v)$ consisting of the
vertex $v$ and $|\cL(v)|$ external edges. In~\cite{CPBMSG} the authors have shown
how the construction of Brownian motions on a finite or semi-infinite interval
by Feller~~\cite{Fe52, Fe54, Fe54a} and It\^o--McKean~\cite{ItMc63, ItMc74} (cf.\ also
\cite{Kn81}) can be extended to the case of single vertex graphs. For the convenience
of the reader, we quickly sketch the method. If $b_v=0$ and
$c_v=1$ then this trivially is a collection of $|\cL(v)|$ many standard Brownian motions
on the real line with absorption at the origin (corresponding to the vertex $v$),
mapped onto the external edges of $\cG(v)$. If $b_v=0$ and $c_v<1$ these
Brownian motions are killed by a jump to $\gD$ after holding the processes at the
origin for an independent exponentially distributed time of rate $a/c$. For
$b_v\ne 0$ one uses a Walsh process on $\cG(v)$ (see \cite{Wa78}, \cite{BaPi89}), and builds
in a time delay as well as killing, both on the scale of the local time at the
vertex. With appropriately chosen parameters for these two mechanisms,
theorem~5.7 in~\cite{CPBMSG} states that the so constructed process $X^v$ is a Brownian
motion on $\cG(v)$ such that its generator is the $1/2$ times the Laplace operator acting
on $f\in C^2_0(\cG(v))$ with boundary conditions at the vertex $v$ given by~\eqref{eq1iii}.

Next we build the graph $\cG$ by successively connecting appropriately chosen
external edges of the single vertex graphs $\cG(v)$, $v\in V$, as in
section~\ref{ssect3i}. Consider the stochastic process $X$ which is successively
constructed from the Brownian motions $X^v$ as in
subsections~\ref{ssect3ii}--\ref{ssect3iv}. Theorem~\ref{thm3xiii} states that $X$
is a Brownian motion on $\cG$ which is such that its generator has a domain which is
characterized by the same boundary conditions at each vertex $v\in V$ as for the
single vertex graphs $\cG(v)$. Therefore, $X$ is a Brownian motion on $\cG$ as in
the statement of theorem~\ref{thm1ii}, whose proof is therefore complete.

\section{Discussion of Tadpoles}	\label{sect5}
Suppose that $\cG$ is a metric graph which has one tadpole $i_t$ connected to a
vertex $v\in V$. That is, $v$ is simultaneously the initial and final vertex of $i_t$:
$\p(i_t)=(v,v)$. Figure~\ref{tad_i} shows a metric graph with a tadpole attached to
the vertex $v$.
\begin{figure}[ht]
\begin{center}
    \includegraphics[scale=.8]{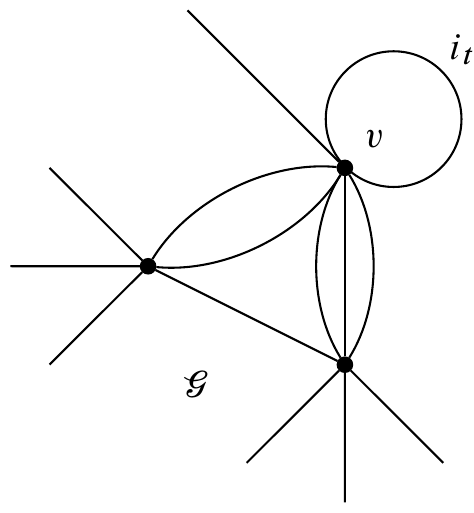}\\
    \caption{A metric graph with a tadpole $i_t$ attached to $v$.} \label{tad_i}
\end{center}
\end{figure}
Let $b_t$ be the length of $i_t$. Assume furthermore, that
we are given data $a$, $b$, $c$ as in equations~\eqref{eq1i}, \eqref{eq1ii}.
We want to construct a Brownian motion $X$ on $\cG$ the implementing boundary
conditions corresponding to these data.

Let $\cG_1$ be the metric graph obtained from $\cG$ by replacing the tadpole by two
external edges $e_1$, $e_2$, incident with $v$. Construct a Brownian motion $X_1$
on $\cG_1$ corresponding to the data $a$, $b$, $c$ as above.

Consider the real line $\R$ as a single vertex graph $\cG_2$ with the origin as the
vertex $v_0$, and edges $l_1$, $l_2$ which are isomorphic to $[0,+\infty)$,
$(-\infty, 0]$. Take a Walsh process $X_2$ on this graph which with probability
$1/2$ chooses either edge for the next excursion when at the origin. Then $X_2$ is
just a ``skew'' Brownian motion as in~\cite[p.~115]{ItMc74} which actually is not
skew. That is, it is equivalent to a standard Brownian motion on the real line.

Now join $\cG_1$ and $\cG_2$ by connecting the pairs $(e_1,l_1)$, $(e_2,l_2)$ via
two new internal edges of length $b_t/2$. Denote the resulting metric graph by $\hat
\cG$. Figure~\ref{tad_ii} shows this construction.
\begin{figure}[ht]
\begin{center}
    \includegraphics[scale=.8]{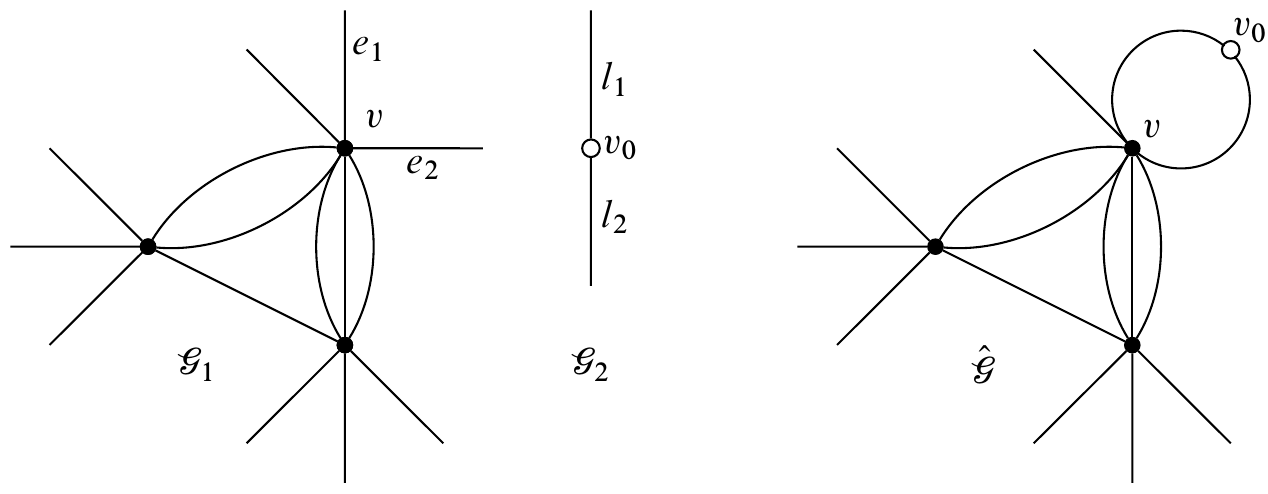}\\
    \caption{$\hat \cG$ constructed from $\cG_1$ and $\cG_2$.} \label{tad_ii}
\end{center}
\end{figure}
Construct a Brownian motion $\hat X$ on $\hat\cG$ from $X_1$ and $X_2$ as in
section~\ref{sect2}. By construction, $\hat X$ is equivalent to a standard Brownian
motion in every neighborhood of $v_0$ which is small enough such that it does not
include the vertex $v$. Therefore the additional vertex $v_0$ of $\cG$ does not
yield any non-trivial boundary condition. Thus, if we identify the open tadpole edge
$i_t^0$ with the subset of $\cG_2$ isomorphic to $(-b_t/1,b_t/2)$, then we obtain a
Brownian motion $X$ on $\cG$ implementing the desired boundary conditions.

Obviously, any (finite) number of tadpoles can be handled in the same way.

\begin{appendix}	
\section{On the Crossover Times $S_n$}	\label{appA}
We recall from section~\ref{sect2} that in terms of the process $Y$ the crossover
times $S_n$, $n\in\N$, can be described as follows. Let $Y$ start in $\xi\in\cG$.
Then $S_1$ is the hitting time of $V_c\setminus\{\xi\}$ by $Y$. In particular,
$S_1>0$ for all paths of $Y$. For $n\ge 2$, $S_n$ is the hitting time after
$S_{n-1}$ of $V_c\setminus\{K_{n-1}\}$ by $Y$. Since by construction
$Y(S_{n-1})=K_{n-1}$ and the paths of $Y$ are continuous on $[0,\zeta)$, we get
$S_n>S_{n-1}$ for all paths of $Y$. Therefore
\begin{align*}
    S_n &= \inf\,\bigl\{t>S_{n-1},\,Y(t)\in V_c\setminus\{K_{n-1}\}\bigr\}\\
        &= \inf\,\bigl\{t\ge S_{n-1},\,Y(t)\in V_c\setminus\{K_{n-1}\}\bigr\}
\end{align*}
holds.

In this appendix we prove the following

\begin{lemma}   \label{lemA}
For every $n\in\N$, $S_n$ is a stopping time relative to $\cF^Y$.
\end{lemma}

\begin{proof}
Set $S_0=0$, and for $n\in\N$,
\begin{equation*}
    V_n = \begin{cases}
            V_c\setminus\{\xi\},     & \text{if $n=1$},\\[1ex]
            V_c\setminus\{K_{n-1}\}, & \text{otherwise}.
          \end{cases}
\end{equation*}
For $n\in\N$, $r\in\R_+$ define
\begin{align*}
    A_{n,r} &= \{S_n \le r < \zeta\}\\[1ex]
    B_{n,r} &= \bigcap_{m\in\N}\bigcup_{u\in\Q,\,0\le u\le r} \bigl\{S_{n-1}\le u,\,
                            d\bigl(Y(u), V_n\bigr)\le 1/m,\,r < \zeta\bigr\}.
\end{align*}
We claim that for all $n\in\N$, $r\in\R_+$,
\begin{equation}    \label{eqAi}
    A_{n,r}=B_{n,r}.
\end{equation}
To prove this claim, suppose first that $\go\in A_{n,r}$. Then $S_n(\go)$ is finite,
and therefore the set
\begin{equation*}
    \bigl\{t\ge S_{n-1}(\go),\,Y(t,\go)\in V_n(\go)\bigr\}\subset [0,\zeta(\go))
\end{equation*}
is non-empty. Thus there exists a sequence $(u_i,\,i\in\N)$ in this set which decreases
to $S_n(\go)$. Since $Y(\,\cdot\,,\go)$ is continuous
on $[0,\zeta(\go))$, it follows that $Y(S_n(\go),\go)\in V_n(\go)$. By assumption,
$S_n(\go)\le r <\zeta(\go)$, and therefore the continuity of $Y(\,\cdot\,,\go)$
on $[0,\zeta(\go))$ and $S_{n-1}(\go)<S_n(\go)$ imply that for every $m\in\N$
there exists $u\in\Q\cap[S_{n-1}(\go),r]$ so that $d(Y(u,\go),V_n(\go))\le 1/m$.
Hence $\go\in B_{n,r}$.

As for the converse, suppose now that $\go\in B_{n,r}$. Then there exists a sequence
$(u_m,\,m\in\N)$ in $\Q\cap[S_{n-1}(\go),r]$ so that $d(Y(u_m,\go), V_n(\go))$ converges
to zero as $m$ tends to infinity. Since $(u_m,\,m\in\N)$ is bounded we may assume, by
selecting a subsequence if necessary, that $(u_m,\,m\in\N)$ converges to  some
$u\in[S_{n-1}(\go),r]$ as $m\to+\infty$. Thus we find that $Y(u,\go)\in V_n(\go)$,
and therefore $Y(\,\cdot\,,\go)$ hits $V_n(\go)$ in the interval $[S_{n-1}(\go),r]$.
Consequently, $S_n(\go)\le r$, and hence $\go\in A_{n,r}$, concluding the
proof of the claim.

Next we prove by induction that for every $n\in\N$, $S_n$ is an $\cF^Y$--stopping time.
Let $n=1$. By~\eqref{eqAi} for every $r\ge 0$,
\begin{equation*}
    A_{1,r} = \bigcap_{m\in\N}\bigcup_{u\in\Q,\,0\le u\le r} \bigl\{
                            d\bigl(Y(u), V_1\bigr)\le 1/m\}\cap \{r < \zeta\bigr\}.
\end{equation*}
Clearly, $\{r<\zeta\} = \{Y(r)\in\cG\}\in\cF^Y_r$. Moreover, since $V_1$ is a
deterministic set, $d(\,\cdot\,,V_1)$ is measurable from $\cG$ to $\R_+$, and
therefore $\{d(Y(u),V_1)\le 1/m\}\in\cF^Y_u \subset \cF^Y_r$. Hence
$A_{1,r}\in\cF^Y_r$. Let $t\ge 0$, and write
\begin{align*}
    \{S_1 \le t\}
        &= \{S_1<\zeta\le t\} \cup A_{1,t}\\
        &= \Bigl(\bigcup_{r\in\Q,\,0\le r\le t} \{S_1\le r <\zeta\}\cap\{\zeta\le t\}\Bigr)\cup A_{1,t}\\
        &= \Bigl(\bigcup_{r\in\Q,\,0\le r\le t} A_{1,r}\cap\{\zeta\le t\}\Bigr)\cup A_{1,t}.
\end{align*}
Therefore $\{S_1\le t\}\in\cF^Y_t$, and hence $S_1$ is a stopping time relative to
$\cF^Y$.

Now suppose that $n\in\N$, $n\ge 2$, and that $S_{n-1}$ is an $\cF^Y$--stopping
time. We show that for all $r\in\R_+$, $A_{n,r}\in\cF^Y_r$. First we remark that
since $\cG$ is a separable metric space, the metric $d$ on $\cG$ is a measurable
mapping from $(\cG\times\cG, \cB_d\otimes\cB_d)$ to $(\R_+,\cB(\R_+))$. For example,
this follows from Theorem~I.1.10 in~\cite{Pa67}, and the continuity of
$d:\cG\times\cG \to \R_+$ when $\cG\times\cG$ is equipped with the product topology.
Consider $K_{n-1}= Y(S_{n-1})$. Since $Y$ has right continuous paths it is
progressively measurable relative to $\cF^Y$ (e.g.,
\cite[Proposition~I.4.8]{ReYo91}). Thus by Proposition~I.4.9 in~\cite{ReYo91} it
follows that $K_{n-1}$ is $\cF^Y_{S_{n-1}}$--measurable. Consequently on
$\{S_{n-1}\le u\}$, $K_{n-1}$ is $\cF^Y_u$--measurable. Equation~\eqref{eqAi} reads
\begin{equation*}
    A_{n,r} = \bigcap_{m\in\N} \bigcup_{u\in\Q,\,0\le u\le r} \{S_{n-1}\le u\}
                    \cap \bigl\{d\bigl(Y(u),V_c\setminus K_{n-1}\bigr)\le 1/m\bigr\}
                    \cap \{r< \zeta\}.
\end{equation*}
It follows that $A_{n,r}\in\cF^Y_r$, as claimed. But then $\{S_n\le t\}\in\cF^Y_t$
for all $t\in\R_+$ is proved with the same argument as at the end of the discussion
of the case $n=1$.
\end{proof}

\section{Feller Semigroups and Resolvents}  \label{app_FSR}	
In this appendix we give an account of the Feller property of semigroups and
resolvents. The material here seems to be quite well-known, and our presentation of
it owes very much to~\cite{Ra56}, most notably the inversion formula for the Laplace
transform, equation~\eqref{inv_L} in connection with lemma~\ref{lem_B_6}. On the
other hand, we were not able to locate a reference where the results are collected and
stated in the form in which we employ them in the present paper. This applies
in particular to the ``mixed'' forms of the statements (iii)--(vi) in theorem~\ref{thm_B_3},
which we find especially convenient to use in this article. Therefore we also provide
proofs for some of the statements.

Assume that $(E,d)$ is a locally compact separable metric space with Borel
$\gs$--algebra denoted by $\cB(E)$. $B(E)$ denotes the space of bounded measurable
real valued functions on $E$, $\Coe$ the subspace of continuous functions vanishing
at infinity. $B(E)$ and $\Coe$ are equipped with the sup-norm $\|\,\cdot\,\|$.

The following definition is as in~\cite{ReYo91}:
\begin{definition}  \label{def_B_1}
    A \emph{Feller semigroup} is a family $U=(U_t,\,t\ge 0)$ of positive
    linear operators on $\Coe$ such that
    \begin{enum_i}
        \item $U_0=\text{id}$ and $\|U_t\|\le 1$ for every $t\ge 0$;
        \item $U_{t+s} = U_t\comp U_s$ for every pair $s$, $t\ge 0$;
        \item $\lim_{t\downarrow 0} \|U_t f - f\|=0$ for every $f\in \Coe$.
    \end{enum_i}
\end{definition}

Analogously we define
\begin{definition}  \label{def_B_2}
    A \emph{Feller resolvent} is a family $R=(R_\gl,\,\gl>0)$ of positive
    linear operators on $\Coe$ such that
    \begin{enum_i}
        \item $\|R_\gl\|\le \gl^{-1}$ for every $\gl>0$;
        \item $R_\gl - R_\mu = (\mu-\gl) R_\gl\comp R_\mu$ for every
                pair $\gl$, $\mu>0$;
        \item $\lim_{\gl\to\infty}\|\gl R_\gl f - f\|=0$ for every $f\in\Coe$.
    \end{enum_i}
\end{definition}

In the sequel we shall focus our attention on semigroups $U$ and resolvents $R$
associated with an $E$--valued Markov process, and which are \emph{a priori} defined
on $B(E)$. (In our notation, we shall not distinguish between $U$ and $R$ as defined
on $B(E)$ and their restrictions to $\Coe$.)

Let $X = (X_t,\,t\ge 0)$ be a Markov process with state space $E$, and let
$(P_x,\,x\in E)$ denote the associated family of probability measures on some measurable
space $(\gO,\cA)$ so that $P_x(X_0=x) = 1$. $E_x(\,\cdot\,)$ denotes the
expectation with respect to $P_x$. We assume throughout that for every $f\in B(E)$
the mapping
\begin{equation*}
    (t,x) \mapsto E_x\bigl(f(X_t)\bigr)
\end{equation*}
is measurable from $\R_+\times E$ into $\R$. The semigroup $U$ and resolvent $R$
associated with $X$ act on $B(E)$ as follows. For $f\in B(E)$, $x\in E$, $t\ge 0$,
and $\gl>0$ set
\begin{align}
    U_t f(x)   &= E_x\bigl(f(X_t)\bigr),  \label{eq_B_1}\\
    R_\gl f(x) &= \int_0^\infty e^{-\gl t} U_t f (x)\,dt.  \label{eq_B_2}
\end{align}
Property~(i) of Definitions~\ref{def_B_1} and \ref{def_B_2} is obviously satisfied. The
semigroup property, (ii) in Definition~\ref{def_B_1}, follows from the Markov property
of $X$, and this in turn implies the resolvent equation, (ii) of
Definition~\ref{def_B_2}. Moreover, it follows also from the Markov property of $X$
that the semigroup and the resolvent commute. On the other hand, in general neither
the property that $U$ or $R$ map $\Coe$ into itself, nor the strong continuity
property (iii) in Definitions~\ref{def_B_1}, \ref{def_B_2} hold true on $B(E)$ or on
$\Coe$.

If $W$ is a subspace of $B(E)$ the resolvent equation shows that the image of $W$
under $R_\gl$ is independent of the choice of $\gl>0$, and in the sequel we shall
denote the image by $RW$. Furthermore, for simplicity we shall write $U\Coe\subset
\Coe$, if $U_t f\in\Coe$ for all $t\ge 0$, $f\in\Coe$.

\begin{theorem} \label{thm_B_3}
The following statements are equivalent:
\begin{enum_i}
    \item $U$ is Feller.
    \item $R$ is Feller.
    \item $U\Coe\subset\Coe$, and for all $f\in\Coe$, $x\in E$,
            $\lim_{t\downarrow 0} U_t f(x) = f(x)$.
    \item $U\Coe\subset\Coe$, and for all $f\in\Coe$, $x\in E$,
            $\lim_{\gl\rightarrow \infty} \gl R_\gl f(x) = f(x)$.
    \item $R\Coe\subset\Coe$, and for all $f\in\Coe$, $x\in E$,
            $\lim_{t\downarrow 0} U_t f(x) = f(x)$.
    \item $R\Coe\subset\Coe$, and for all $f\in\Coe$, $x\in E$,
            $\lim_{\gl\rightarrow \infty} \gl R_\gl f(x) = f(x)$.
\end{enum_i}
\end{theorem}

\begin{remark}	
The equivalence of statements~(i) and~(ii) has been shown in~\cite[No.~81, p.~291]{DeMe88}
based on an application of the Hille--Yosida--theorem.
\end{remark}

We prepare a sequence of lemmas. The first one follows directly from the
dominated convergence theorem:

\begin{lemma}   \label{lem_B_4}
Assume that for $f\in B(E)$, $U_t f \rightarrow f$ as $t\downarrow 0$. Then $\gl
R_\gl f \rightarrow f$ as $\gl\to+\infty$.
\end{lemma}

\begin{lemma}   \label{lem_B_5}
The semigroup $U$ is strongly continuous on $RB(E)$.
\end{lemma}

\begin{proof}
If strong continuity at $t=0$ has been shown, strong continuity at $t>0$ follows
from the semigroup property of $U$, and the fact that $U$ and $R$ commute. Therefore
it is enough to show strong continuity at $t=0$.

Let $f\in B(E)$, $\gl>0$, $t>0$, and consider for $x\in E$ the following computation
\begin{align*}
    U_t R_\gl f(x) &- R_\gl f(x)\\
        &= \int_0^\infty e^{-\gl s} E_x\bigl(f(X_{t+s})\bigr)\,ds
                - \int_0^\infty e^{-\gl s} E_x\bigl(f(X_s)\bigr)\,ds\\
        &= e^{\gl t} \int_t^\infty e^{-\gl s} E_x\bigl(f(X_s)\bigr)\,ds
                - \int_0^\infty e^{-\gl s} E_x\bigl(f(X_s)\bigr)\,ds\\
        &= \bigl(e^{\gl t}-1\bigr) \int_t^\infty e^{-\gl s} E_x\bigl(f(X_s)\bigr)\,ds
                - \int_0^t e^{-\gl s} E_x\bigl(f(X_s)\bigr)\,ds\\
\end{align*}
where we used Fubini's theorem and the Markov property of $X$. Thus we get the
following estimation
\begin{align*}
    \bigl\|U_t R_\gl f - R_\gl f\bigr\|
        &\le \biggl(\bigl(e^{\gl t} - 1\bigr)\int_t^\infty e^{-\gl s}\,ds
                +\int_0^t e^{-\gl s}\,ds\biggl)\, \|f\|\\
        &= \frac{2}{\gl}\, \bigl(1-e^{-\gl t}\bigr)\,\|f\|,
\end{align*}
which converges to zero as $t$ decreases to zero.
\end{proof}

For $\gl>0$, $t\ge 0$, $f\in B(E)$, $x\in E$ set
\begin{equation}    \label{inv_L}
    U^\gl_t f(x) = \sum_{n=1}^\infty \frac{(-1)^{n+1}}{n!}
                        \,n\gl\, e^{n\gl t}\, R_{n\gl} f(x).
\end{equation}
Observe that, because of $n\gl\|R_{n\gl} f\| \le \|f\|$, the last sum converges in
$B(E)$.

For the proof of the next lemma we refer the reader
to~\cite[p.~477~f]{Ra56}:

\begin{lemma}   \label{lem_B_6}
For all $t\ge 0$, $f\in RB(E)$, $U^\gl_t f$ converges in $B(E)$ to $U_t f$ as
$\gl$ tends to infinity.
\end{lemma}

\begin{lemma}   \label{lem_B_7}
If $U_t\Coe \subset \Coe$ for all $t\ge 0$, then $R_\gl\Coe\subset \Coe$, for all $\gl>0$.
If $R_\gl\Coe\subset \Coe$, for some $\gl>0$, and $R_\gl\Coe$ is dense in $\Coe$, then
$U_t\Coe \subset \Coe$ for all $t\ge 0$.
\end{lemma}

\begin{proof}
Assume that  $U_t\Coe \subset \Coe$ for all $t\ge 0$, let $f\in\Coe$, $x\in E$, and
suppose that $(x_n,\,n\in\N)$ is a sequence converging in $(E,d)$ to $x$. Then a
straightforward application of the dominated convergence theorem shows that for
every $\gl>0$, $R_\gl f(x_n)$ converges to $R_\gl f(x)$. Hence $R_\gl f\in \Coe$.

Now assume that that $R_\gl\Coe\subset \Coe$, for some and therefore for all
$\gl>0$, and that $R_\gl\Coe$ is dense in $\Coe$. Consider $f\in R\Coe$, $t>0$, and
for $\gl>0$ define $U^\gl_t f$ as in equation~\eqref{inv_L}. Because
$R_{n\gl}f\in\Coe$ and the series in formula~\eqref{inv_L} converges uniformly in
$x\in E$, we get $U^\gl_t f\in\Coe$. By lemma~\ref{lem_B_6}, we find that $U^\gl_t
f$ converges uniformly to $U_t f$ as $\gl\to+\infty$. Hence $U_t f\in\Coe$. Since
$R\Coe$ is dense in $\Coe$, $U_t$ is a contraction and $\Coe$ is closed, we get that
$U_t\Coe\subset\Coe$ for every $t\ge 0$.
\end{proof}

The following lemma is proved as a part of Theorem~17.4 in~\cite{Ka97}
(cf.\ also the proof of Proposition~2.4 in~\cite{ReYo91}).

\begin{lemma}   \label{lem_B_8}
Assume that $R\Coe\subset \Coe$, and that for all $x\in E$, $f\in\Coe$,
$\lim_{\gl\to\infty} \gl R_\gl f(x) = f(x)$. Then $R\Coe$ is dense in $\Coe$.
\end{lemma}

If for all $f\in\Coe$, $x\in E$, $U_t f(x)$ converges to $f(x)$ as $t$ decreases to
zero, then similarly as in the proof of lemma~\ref{lem_B_4} we get that $\gl R_\gl
f(x)$ converges to $f(x)$ as $\gl\to+\infty$. Thus we obtain the following

\begin{corollary}   \label{cor_B_9}
Assume that $R\Coe\subset \Coe$, and that for all $x\in E$, $f\in\Coe$,
$\lim_{t\downarrow 0} U_t f(x) = f(x)$. Then $R\Coe$ is dense in $\Coe$.
\end{corollary}

Now we can come to the

\begin{proof}[Proof of theorem~\ref{thm_B_3}]
We show first the equivalence of statements~(i), (ii), (iv), and (vi):

\vspace{.5\baselineskip}\noindent
``(i)\Ra(ii)'' Assume that $U$ is Feller. From lemma~\ref{lem_B_7} it follows that
$R_\gl\Coe\subset\Coe$, $\gl>0$. Let $f\in\Coe$. Since $U$ is strongly continuous on
$\Coe$, lemma~\ref{lem_B_4} implies that $\gl R_\gl f$ converges to $f$ as $\gl$ tends
to $+\infty$. Hence $R$ is Feller.

\vspace{.25\baselineskip}\noindent
``(ii)\Ra(vi)'' This is trivial.

\vspace{.25\baselineskip}\noindent
``(vi)\Ra(iv)'' By lemma~\ref{lem_B_8}, $R\Coe$ is dense in $\Coe$, and therefore
lemma~\ref{lem_B_7} entails that $U\Coe\subset\Coe$.

\vspace{.25\baselineskip}\noindent
``(iv)\Ra(i)'' By lemmas~\ref{lem_B_7} and \ref{lem_B_8},  $R\Coe$ is dense in
$\Coe$, and therefore by lemma~\ref{lem_B_5} $U$ is strongly continuous on $\Coe$. Thus
$U$ is Feller.

\vspace{.25\baselineskip}
Now we prove the equivalence of~(i), (iii), and (v):

\vspace{.25\baselineskip}\noindent
``(i)\Ra(iii)'' This is trivial.

\vspace{.25\baselineskip}\noindent
``(iii)\Ra(v)'' This follows directly from Lemma~\ref{lem_B_7}.

\vspace{.25\baselineskip}\noindent
``(v)\Ra(i)'' By corollary~\ref{cor_B_9}, $R\Coe$ is dense in $\Coe$, hence it follows
from lem\-ma~\ref{lem_B_7} that $U\Coe\subset\Coe$. Furthermore, lemma~\ref{lem_B_5}
implies the strong continuity of $U$ on $R\Coe$, and by density therefore on $\Coe$.
(i) follows.
\end{proof}
\end{appendix}

\providecommand{\bysame}{\leavevmode\hbox to3em{\hrulefill}\thinspace}
\providecommand{\MR}{\relax\ifhmode\unskip\space\fi MR }
\providecommand{\MRhref}[2]{%
  \href{http://www.ams.org/mathscinet-getitem?mr=#1}{#2}
}
\providecommand{\href}[2]{#2}

\end{document}